\documentclass[11pt]{article}
\usepackage[tbtags]{amsmath}
\usepackage{amssymb}
\usepackage{amsthm}
\usepackage[misc]{ifsym}
\usepackage{cases}
\usepackage{mathrsfs}
\usepackage{graphicx}
\usepackage{subfigure}
\graphicspath{{./image/}}

\newcommand{\esssup}{\mathop{\mathrm{esssup}}}

\numberwithin{equation}{section}
\normalsize

\title{\bf Stochastic Linear-Quadratic Stackelberg Differential Game with Asymmetric Informational Uncertainties: Robust Optimization Approach\thanks{This work is supported by National Key R\&D Program of China (2022YFA1006104), National Natural Science Foundation of China (12271304), and Shandong Provincial Natural Science Foundations (ZR2022JQ01, ZR2020ZD24).}}

\author{\normalsize  Na Xiang\thanks{\it School of Mathematics, Shandong University, Jinan 250100, P.R. China, E-mail: 202211967@mail.sdu.edu.cn} , Jingtao Shi\thanks{Corresponding author, \it School of Mathematics, Shandong University, Jinan 250100, P.R. China, E-mail: shijingtao@sdu.edu.cn}}

\newtheorem{mypro}{Proposition}[section]
\newtheorem{mythm}{Theorem}[section]
\newtheorem{mydef}{Definition}[section]
\newtheorem{mylem}{Lemma}[section]
\newtheorem{Remark}{Remark}[section]
\newtheorem{Example}{Example}[section]

\begin{document}

\maketitle
	
\noindent{\bf Abstract:} This paper is concerned with a two-person zero-sum indefinite stochastic linear-quadratic Stackelberg differential game with asymmetric informational uncertainties, where both the leader and follower face different and unknown disturbances. We take a robust optimization approach and soft-constraint analysis, a min-max stochastic linear-quadratic optimal control problem is solved by the follower firstly. Then, the leader deal with a max-min stochastic linear-quadratic optimal control problem of forward-backward stochastic differential equations in an augmented space. State feedback representation of the robust Stackelberg equilibrium is given in a more explicit form by decoupling technique, via some Riccati equations.
	
\vspace{2mm}

	
	
\section{Introduction}
	
The Stackelberg game, also known as the leader-follower game, is widely involved in many fields, such as economics, finance and engineering. The research of Stackelberg game was first introduced by von Stackelberg \cite{Stackelberg34}, who proposed the concept of a hierarchical solution in static competitive economics. Bagchi and Ba\c{s}ar \cite{Bagchi-Basar81} initially studied the {\it stochastic linear-quadratic} (SLQ) Stackelberg differential game, while the diffusion term of the state equation does not contain the state and control variables, and the Stackelberg solution was expressed in terms of Riccati equations. Yong \cite{Yong02} formulated a general framework, in which the coefficients could be random, the control variables could enter the diffusion term of the state equation and the weight matrices of the control variable in the cost functionals need not to be positive definite. The leader's problem was described as an optimal control problem of {\it forward-backward stochastic differential equations} (FBSDEs). By a decoupling method, the open-loop solution can be represented as a state feedback form, provided the associated stochastic Riccati equation is solvable. From then on, there has been extensive research on stochastic Stackelberg differential game problems. Let us mention a few related to this article. Bensoussan et al. \cite{Bensoussan-Chen-Sethi15} established maximum principles for the stochastic Stackelberg differential game with the control-independent diffusion term in different information structures. Shi et al. \cite{Shi-Wang-Xiong16, Shi-Wang-Xiong17} investigated SLQ Stackelberg differential games with asymmetric information, and Li et al. \cite{Li et al21} solved a similar problem by a layered calculation method. Li and Yu \cite{Li-Yu18} characterized the unique equilibrium of a nonzero-sum SLQ Stackelberg differential game with multilevel hierarchy. Moon and Ba\c{s}ar \cite{Moon-Basar18} considered an SLQ Stackelberg {\it mean field game} (MFG) with one leader and arbitrarily large number of followers by fixed-point approach, while recently Wang \cite{Wang24} solved it by a direct approach. Lin et al. \cite{Lin-Jiang-Zhang19} considered LQ Stackelberg differential game of mean-field type stochastic systems. Moon and Yang \cite{Moon20} discussed the time-consistent open-loop solutions for time-inconsistent SLQ mean-field type Stackelberg differential game. Zheng and Shi \cite{Zheng-Shi22} researched a Stackelberg stochastic differential game with asymmetric noisy observations. Feng et al. \cite{Feng-Hu-Huang24} investigated the relationships between zero-sum SLQ Nash and Stackelberg differential game, local versus global information.
	
The research of zero-sum SLQ Nash differential games is very fruitful, such as Mou and Yong \cite{Mou-Yong06}, Sun and Yong \cite{Sun-Yong14}, Yu \cite{Yu15} and the references therein. For zero-sum SLQ Stackelberg differential games, Lin et al. \cite{Lin-Zhang-Siu12} formulated an optimal portfolio selection problem with model uncertainty as a zero-sum stochastic Stackelberg differential game between the investor and the market. Sun et al. \cite{Sun-Wang-Wen23} studied a zero-sum SLQ Stackelberg differential game, where the state coefficients and cost weighting matrices are deterministic. Wu et al. \cite{Wu-Xiong-Zhang24} investigated a zero-sum SLQ Stackelberg differential game with jumps, with random coefficients. Aberkane and Dragan \cite{Aberkane-Dragan23} investigated the zero-sum SLQ mean-field type game with a leader-follower structure.
	
For traditional research on Stackelberg games, most works assume that model parameters can be accurately calibrated without information loss. The decision-makers specify the true probability law of the state process. Due to the basic facts such as parameter uncertainty and model ambiguity, etc., these perfect model assumptions are sometimes noneffective. Thus, it is very meaningful to investigate {\it model uncertainty}. Aghassi and Bertsimas \cite{Aghassi-Bertsimas06} introduced a robust optimization equilibrium, where the players use a robust optimization approach to contend with payoff uncertainty. Hu and Fukushima \cite{Hu-Fikushima13} focused on a class of multi-leader single-follower games under uncertainty with some special structure, where the follower's problem contains only equality constraints. Jimenez and Poznyak \cite{Jimenez-Poznyak06} considered multi-persons LQ differential game with bounded disturbances and uncertainties in finite time. van den Broek et al. \cite{Broek et al03} analyzed the robust equilibria in dynamic games, where players are looking for robustness and take model uncertainty explicitly into account in decisions, and considered the soft-constrained and hard-bounded cases. Huang and Huang \cite{Huang-Huang13} initially considered the robust MFG with a hard constraint, and \cite{Huang-Huang17} studied MFG with model uncertainty, where the uncertainty disturbance is deterministic on the drift term only. Moon and Ba\c{s}ar \cite{Moon-Basar17} studied a risk-sensitive SLQ robust MFG. Xie et al. \cite{Xie-Wang-Huang21} researched a robust SLQ mean field social control by a direct approach. Wang et al. \cite{Wang-Huang-Zhang21} considered  social optimal control of mean field SLQ with uncertainty which is an uncertain drift, while Huang et al. \cite{Huang-Wang-Yong21} studied mean field SLQ social optimum control with volatility uncertainty. Jia et al. \cite{Jia et al23} focused on a robust backward SLQ differential game and team. Feng et al. \cite{Feng-Qiu-Wang24} dealt with an LQ large population system with uncertain volatilities.
	
However, research about Stackelberg differential games with model uncertainty is very lacking. To the best of our knowledge, only Huang et al. \cite{Huang-Wang-Wu22} investigated two types of drift stochastic uncertainties in Stackelberg differential game, where information uncertainty and temporal uncertainty are connected with soft-constraint and hard-constraint min-max control, respectively.

In this paper, we consider a zero-sum SLQ Stackelberg differential game with asymmetric informational uncertainties for two players. Specially, we focus on the drift uncertainty by adding two unknown $L^2$-disturbance $f_1$, $f_2$ to characterize the different model uncertainties and represent the influence from the common environment for the leader and the follower, respectively. This has practical significance. For example, in the market there exists information asymmetry between suppliers and purchasers, who cannot know all the information about the market. The leader, as the dominant player, possesses superior informational ability, so the leader may access more information than the follower. Due to the limited personal abilities and complicated information environment, the follower cannot completely access all information. Thus, $f_1$ is fully observed by the leader and $f_2$ is formalized as an unknown disturbance. Meanwhile, $f_1$ and $f_2$ are both unknown to the follower. We address model uncertainty by the soft-constraint approach (\cite{Broek et al03}, \cite{Huang-Huang17}, \cite{Huang-Wang-Wu22}), by using a cost penalty term for the disturbance and viewing it as an adversarial player. Our main contributions are summarized as follows.

$\bullet$ We formulate a class of zero-sum SLQ Stackelberg differential games where the leader and the follower face different uncertainty sources having hierarchical relationship, and utilize the robust optimization approach to solve the corresponding soft-constraint min-max control and soft-constraint max-min control problems.

$\bullet$ Via convex optimization approach, the ``worst case"  disturbances and two optimal controls for the players are derived.

$\bullet$ The weighting matrix of the follower is allowed to be indefinite, then the leader's problem could still be indefinite.

$\bullet$ By applying the decoupling technique, with the solution to Riccati equations, the state feedback representation of the robust Stackelberg equilibrium is obtained.
	
The rest of this article is organized as follows. Section 2 introduces some preliminary notations and the formulation of zero-sum Stackelberg differential games. Section 3 discusses the informational uncertainty and related robust strategy design for the follower. Section 4 solves the augmented LQ optimal control problem of FBSDE for the leader, to ensure the robust Stackelberg equilibrium. Section 5 provides the state feedback representation of the robust Stackelberg equilibrium. The results are applied to a production supply problem of two producers in the market. Section 6 concludes the paper.

\section{Preliminaries and problem formulation}
	
Let ($\Omega,\mathcal{F},\mathbb{F},\mathbb{P}$) be a complete filtered probability space, on which a standard one-dimensional Brownian motion $W=\left\lbrace W(t),0\leqslant t<\infty\right\rbrace$ is defined, where $\mathbb{F} =\left\lbrace \mathcal{F}_t\right\rbrace _{t\geqslant0} $ is the natural filtration of $W\left(\cdot\right)$ augmented by all the $\mathbb{P}$-null sets in $\mathcal{F}$. Let $\mathbb{R}^n$ denote the $n$-dimensional Euclidean space with Euclidean norm $ \left| \cdot \right| $ and inner product $\left\langle \cdot,\cdot\right\rangle $. The transpose of a vector (or a matrix) $x$ is denoted by $x^\top$. $\mbox{Tr}(A)$ denotes the trace of a square matrix $A$. Let $\mathbb{R}^{n\times m}$ be the Hilbert space consisting of all $(n\times m)$-matrices with the inner product $\langle A,B\rangle :=\mbox{Tr}(AB^\top)$ and the norm $\Vert A \Vert:=\langle A,A\rangle^\frac{1}{2}$. Denote the set of symmetric $n\times n$ matrices with real elements by $\mathbb{S}^n$. If $M\in\mathbb{S}^n$ is positive (semi-)definite, we write $M>(\geqslant) 0$. If there exists a constant $\delta>0$ such that $M\geqslant\delta I$ with the identity matrix $I$, we write $M\gg0$.	
	
Consider a finite time horizon $[0,T]$ for a fixed $T>0$. Let $\mathbb{H}$ be a given Hilbert space. The set of $\mathbb{H}$-valued continuous functions is denoted by $C([0,T];\mathbb{H})$. Let
\begin{equation*}
\begin{aligned}
	L_{\mathbb{F}}^p(0,T;\mathbb{H})&:=\biggl\{\phi :[0,T]\times\Omega \mapsto \mathbb{H}\,\Big|\,\phi\mbox{ is }\mathbb{F}\mbox{-progressively measurable} \\
	&\qquad\quad \mbox{with }\left( \mathbb{E}\int_0^T |\phi(s)|^p ds\right) ^\frac{1}{p} < +\infty \biggr\},\quad p\geq 1, \\
	L_{\mathbb{F}}^\infty(0,T;\mathbb{H})&:=\biggl\{\phi :[0,T]\times \Omega\mapsto\mathbb{H}\,\big|\,\phi\mbox{ is }\mathbb{F}\mbox{-progressively measurable} \\
	&\qquad\quad \mbox{with } \esssup_{s\in[0,T]}\esssup_{\omega\in\Omega}|\phi(s)| < +\infty\biggr\},\\
	L_{\mathbb{F}}^2(\Omega;C([0,T];\mathbb{H}))&:=\biggl\{\phi :[0,T]\times\Omega\mapsto\mathbb{H}\,\big|\,\phi\mbox{ is }\mathbb{F}\mbox{-adapted, continuous} \\
	&\qquad\quad \mbox{with }\mathbb{E}\biggl[ \sup_{s\in[0,T]}|\phi(s)|^2\biggr] < +\infty\biggr\}.
\end{aligned}
\end{equation*}
When a stochastic process reduces to be deterministic, $L_{\mathbb{F}}^2(0,T;\mathbb{H}),L_\mathbb{F}^\infty(0,T;\mathbb{H})$ are denoted by $L^2(0,T;\mathbb{H}),L^\infty(0,T;\mathbb{H})$, respectively. For $\phi\in L_{\mathbb{F}}^2(0,T;\mathbb{H})$ (resp., $L^2(0,T;\mathbb{H})$), we denote
\[
\Vert \phi \Vert_{L^2}:=\left( \mathbb{E}\int_0^T|\phi(s)|^2 ds\right) ^\frac{1}{2}\;\left( \mbox{resp.},\left( \int_0^T|\phi(s)|^2 ds\right)  ^\frac{1}{2}\right).
\]
	
Consider the following controlled linear {\it stochastic differential equation} (SDE) on $[0,T]$:
\begin{equation}\label{state}
\left\{
\begin{aligned}
	\mathrm{d}x(t)=&\ \big[ A(t)x(t)+B_1(t)u_1(t)+B_2(t)u_2(t)+f_1(t)+f_2(t)\big]\mathrm{d}t \\
	&+\big[ C(t)x(t)+D_1(t)u_1(t)+D_2(t)u_2(t)+\sigma(t)\big] \mathrm{d}W(t),  \\
	x(0)=&\ \xi,
\end{aligned}
\right.
\end{equation}	
where $A(\cdot)$, $B_1(\cdot)$, $B_2(\cdot)$, $C(\cdot)$, $D_1(\cdot)$, $D_2(\cdot)$, $\sigma(\cdot)$ are deterministic functions of proper dimensions, $f_1(\cdot),f_2(\cdot)$ are $\mathbb{F}$-adapted processes, $\xi\in\mathbb{R}^n$ is the initial state. The process $u_i(\cdot)$ represents the control of Player $i$, which belongs to the following space:
\begin{equation}
\begin{aligned}
	\mathcal{U}_i[0,T]:=\biggl\{ &u_i :[0,T]\times \Omega\mapsto\mathbb{R}^{m_i}\Big|u_i(\cdot) \mbox{ is }\mathbb{F}\mbox{-progressively measurable}\\
	&\quad \mbox{with }\mathbb{E}\int_0^T|u_i(s)|^2 ds < +\infty \biggr\}, \mbox{ for }i=1,2.
\end{aligned}
\end{equation}
The solution $x(\cdot)\equiv x(\cdot;\xi,u_1,u_2,f_1,f_2)\in\mathbb{R}^n$ of (\ref{state}) is called the \emph{state process} corresponding to $\xi$ and $(u_1(\cdot),u_2(\cdot),f_1(\cdot),f_2(\cdot))$. The criterion for the performance of $u_1(\cdot)$ and $u_2(\cdot)$ is given by the following quadratic functional:
\begin{equation}\label{cost}
\begin{aligned}
	 J(\xi;u_1(\cdot),u_2(\cdot))=&\ \mathbb{E}\bigg\{ \int_0^T\Big[\left\langle Q(t)x(t),x(t)\right\rangle +\left\langle R_1(t)u_1(t),u_1(t)\right\rangle\\
	 &\qquad +\left\langle R_2(t)u_2(t),u_2(t)\right\rangle\Big]\mathrm{d}t+\left\langle Gx(T),x(T)\right\rangle \bigg\},
\end{aligned}
\end{equation}
where $G\in\mathbb{S}^n$, $Q(\cdot)$ and $R_i(\cdot)$ $(i=1,2)$ are symmetric matrix-valued functions.
	
In the Stackelberg game without asymmetric informational uncertainties, player 2 is the leader, who announces his/her control $u_2(\cdot)$ first, and player 1 is the follower, who chooses his/her control as the optimal response accordingly. The criterion functional $J(\xi;u_1(\cdot),u_2(\cdot))$ is regarded as the loss of player 1 and the gain of player 2. For any choice $u_2(\cdot)\in\mathcal{U}_2[0,T]$ of the leader and a fixed initial state $\xi\in\mathbb{R}^n$, the follower would like to choose a $\bar{u}_1(\cdot)\equiv\bar{u}_1[\xi,u_2(\cdot)](\cdot)\in\mathcal{U}_1[0,T]$ such that $J(\xi;u_1(\cdot),u_2(\cdot))$ is minimized. Knowing this, the leader wishes to choose some $\bar{u}_2(\cdot)\in\mathcal{U}_2[0,T]$ so that $J(\xi;\bar{u}_1[\xi,u_2(\cdot)](\cdot),u_2(\cdot))$ is maximized. We refer to such a problem as a \textit{zero-sum SLQ Stackelberg differential game}. The main objective of two players is to find the \textit{Stackelberg equilibrium} of the game, mathematically defined as follows.
	
\begin{mydef}\label{def2.1}
A control pair $(\bar{u}_1(\cdot),\bar{u}_2(\cdot))\in \mathcal{U}_1[0,T] \times\mathcal{U}_2[0,T]$ is called a \textit{Stackelberg equilibrium} of the \textit{zero-sum SLQ Stackelberg differential game} for \textit{the initial state} $\xi$, if
\begin{equation*}
\begin{aligned}
&\inf_{u_1(\cdot)\in\,\mathcal{U}_1[0,T]}J(\xi;u_1(\cdot),\bar{u}_2(\cdot))=J(\xi;\bar{u}_1(\cdot),\bar{u}_2(\cdot))\\
&=\sup_{u_2(\cdot)\in\,\mathcal{U}_2[0,T]} \inf_{u_1(\cdot)\in\,\mathcal{U}_1[0,T]}J(\xi;u_1(\cdot),u_2(\cdot)).
\end{aligned}
\end{equation*}
\end{mydef}
	
We assume that the coefficients of state equation (\ref{state}) and the weighting matrices in cost functional (\ref{cost}) satisfy the following conditions: for $i=1,2$,

{\bf(A1)} $A(\cdot), C(\cdot) \in L^\infty(0,T;\mathbb{R}^{n\times n})$, $B_i(\cdot),D_i(\cdot) \in L^\infty(0,T;\mathbb{R}^{n\times m_i})$, $f_i(\cdot) \in L_{\mathbb{F}}^2(0,T;\mathbb{R}^n)$,
    $\sigma(\cdot) \in L^\infty(0,T;\mathbb{R}^{n})$, $\xi\in\mathbb{R}^n$.

{\bf(A2)} $Q(\cdot) \in L^\infty(0,T;\mathbb{S}^{n})$, $R_i(\cdot) \in L^\infty(0,T;\mathbb{S}^{m_i})$, $G \in\mathbb{S}^n$.
	
Let {\bf(A1)} hold. For any $(\xi,u_1(\cdot),u_2(\cdot),f_1(\cdot),f_2(\cdot)) \in \mathbb{R}^n \times\mathcal{U}_1[0,T] \times\mathcal{U}_2[0,T] \times L_{\mathbb{F}}^2(0,T;\mathbb{R}^{2n})$, there exists a unique strong solution $x(\cdot)\equiv x(\cdot;\xi,u_1,u_2,f_1,f_2)\in L_{\mathbb{F}}^2\left( \Omega;C([0,T];\mathbb{R}^n)\right)$. Then under assumption {\bf(A2)}, the cost functional (\ref{cost}) is well-defined.

In this article, we may suppress many time indices $t$ if it causes no confusion.
	
\section{Robust Stackelberg strategy for the follower}
	
We first introduce the asymmetric informational uncertainties that may connect to some information transmission distinctions between the leader and the follower. In (\ref{state}), the stochastic process $f_i(\cdot)\in L_{\mathbb{F}}^2(0,T;\mathbb{R}^n)$ $(i=1,2)$ are unknown disturbances to characterize the model uncertainty and represent the influence from the common environment for decision-making. Here, the $f_i$ is set to be random, the disturbance, the adversarial will use the sample path information of Brownian motion to play against the leader and the follower, instead of considering a deterministic disturbance. In this article, we think that the leader and the follower confront the possible modeling disturbance in dynamic evolution. As mentioned in the introduction, $f_1$ is fully observed by the leader and $f_2$ is formalized as an unknown disturbance. Meanwhile, $f_1$ and $f_2$ are both unknown to the follower. For any given $u_2(\cdot)$, the follower should make his best response but now faces the uncertain disturbance $f_1$ and $f_2$, which enters the state equation (\ref{state}). The follower may adopt some soft-constraint analysis and his cost functional is
\begin{equation}\label{cost_F1}
\begin{aligned}
	&J_f(f(\cdot);\xi,u_1(\cdot),u_2(\cdot))\\
    =&\ \mathbb{E}\bigg\{\int_0^T \Big[\langle Q(t)x(t),x(t)\rangle +\langle R_1(t)u_1(t),u_1(t)\rangle+\langle R_2(t)u_2(t),u_2(t)\rangle \\
	&\qquad\qquad -\frac{\alpha}{2}\langle R_0(t)f(t),f(t)\rangle \Big] \mathrm{d}t+\langle Gx(T),x(T)\rangle \bigg\},
\end{aligned}
\end{equation}
where the constant $\alpha > 0$ is called the $\textit{attenuation}$ parameter of soft-constraint (see \cite{Huang-Huang17}, \cite{Huang-Wang-Wu22}), and the matrix-valued function $R_0(\cdot)>0$, with $f(\cdot):=f_1(\cdot)+f_2(\cdot)$ in this section of the follower's problem. The follower considers the worse-case analysis and the following inf-sup problem:
	
Problem \textbf{(Inf-Sup)}:$$\inf_{u_1(\cdot)\in\,\mathcal{U}_1[0,T]}\sup_{f(\cdot)\in L_{\mathbb{F}}^2(0,T;\mathbb{R}^n)}J_f(f(\cdot);\xi,u_1(\cdot),u_2(\cdot)),\mbox{ for any $\xi \in\mathbb{R}^n$,  }u_2(\cdot)\in\mathcal{U}_2[0,T].$$

In the following we solve the follower's problem in two steps. We need first tackle the \textit{inner} LQ maximizing problem for disturbance:

Problem \textbf{(LQ-F1)}: To find $\bar{f}(\cdot)\in L_{\mathbb{F}}^2(0,T;\mathbb{R}^n)$ such that
$$
J_f\big(\bar{f}(\cdot);\xi,u_1(\cdot),u_2(\cdot)\big)=\sup_{f(\cdot)\in L_{\mathbb{F}}^2(0,T;\mathbb{R}^n)}J_f(f(\cdot);\xi,u_1(\cdot),u_2(\cdot)),
$$
subject to $\xi \in\mathbb{R}^n$, $(u_1(\cdot),u_2(\cdot))\in\mathcal{U}_1[0,T] \times\mathcal{U}_2[0,T]$ and the state equation (\ref{state}) with $f(\cdot)\equiv f_1(\cdot)+f_2(\cdot)$.

Define a map $\bar{\alpha}_1:\mathbb{R}^{n}\times\mathcal{U}_1[0,T]\times\mathcal{U}_2[0,T]\mapsto L_{\mathbb{F}}^2(0,T;\mathbb{R}^{n})$. We may denote the maximizer $\bar{f}(\cdot)=\bar{\alpha}_1\left[ \xi,u_1(\cdot),u_2(\cdot)\right] (\cdot) $ to emphasize its dependence on the triple $(\xi,u_1(\cdot),u_2(\cdot))$. Denote
$$
J^{wo}(u_1(\cdot);\xi,u_2(\cdot)):=J_f\left( \bar{\alpha}_1\left[ \xi,u_1(\cdot),u_2(\cdot)\right] (\cdot);\xi,u_1,u_2\right).
$$
Here, the superscript ``wo" stands for the worst-case disturbance of the follower (Korn and Menkens \cite{Korn-Menkens05}). Given the mapping $\bar{\alpha}_1$, the follower needs solve the \textit{outer} LQ minimizing problem:
	
Problem \textbf{(LQ-F2)}:
To find $\bar{u}_1(\cdot)\in \mathcal{U}_1[0,T]$ such that
$$
J^{wo}(\bar{u}_1(\cdot);\xi,u_2(\cdot))=\inf_{u_1(\cdot)\in\,\mathcal{U}_1[0,T]}J^{wo}(u_1(\cdot);\xi,u_2(\cdot)),
$$
subject to $\xi \in\mathbb{R}^n$, $u_2(\cdot)\in\mathcal{U}_2[0,T]$. Similarly, denote the minimizer $\bar{u}_1(\cdot)=\bar{\alpha}_2\left[ \xi,u_2(\cdot)\right] (\cdot)$ to indicate its dependence on $\xi$ and $u_2(\cdot)$, where $\bar{\alpha}_2:\mathbb{R}^{n}\times\mathcal{U}_2[0,T]\mapsto \mathcal{U}_1[0,T]$ is a map.
	
\subsection{Problem (\textbf{LQ-F1})}
	
We first deal with the inner LQ maximizing problem. Problem \textbf{(LQ-F1)} can be rewritten as an equivalent problem:
	
Problem \textbf{(LQ-F1a)}:
To find $\bar{f}(\cdot)\in L_{\mathbb{F}}^2(0,T;\mathbb{R}^{n})$ such that
$$
J'_f(\bar{f}(\cdot);\xi,u_1(\cdot),u_2(\cdot))=\inf_{f(\cdot)\in L_{\mathbb{F}}^2(0,T;\mathbb{R}^{n})}J'_f(f(\cdot);\xi,u_1(\cdot),u_2(\cdot)),
$$
subject to $\xi \in\mathbb{R}^n$, $(u_1(\cdot),u_2(\cdot))\in\mathcal{U}_1[0,T] \times\mathcal{U}_2[0,T]$ and the state equation (\ref{state}) with $f(\cdot)\equiv f_1(\cdot)+f_2(\cdot)$, where
\begin{equation*}
\begin{aligned}
J'_f(f(\cdot);\xi,u_1(\cdot),u_2(\cdot))&:=\mathbb{E}\bigg\{ \int_0^T\left[-\left\langle Q(t)x(t),x(t)\right\rangle +\frac{\alpha}{2}\left\langle R_0(t)f(t),f(t)\right\rangle\right] \mathrm{d}t\\
&\qquad -\left\langle Gx(T),x(T)\right\rangle \bigg\}.
\end{aligned}
\end{equation*}

\begin{mydef}
Let $F(g)$ be a \textit{real-valued functional} of $g\in L_{\mathbb{F}}^2(0,T;\mathbb{R}^n)$. If $F(g)\geqslant 0$ for all $g\in L_{\mathbb{F}}^2(0,T;\mathbb{R}^n)$, $F(\cdot)$ is said to be \textit{positive semidefinite}. If furthermore, $F(g)> 0$ for all $g\neq 0$, $F$ is said to be \textit{positive definite}.
\end{mydef}

\begin{mylem}\label{lemma3.1}
Let (A1)-(A2) hold. For any $\xi\in\mathbb{R}^n$ and $(u_1(\cdot),u_2(\cdot))\in \mathcal{U}_1[0,T] \times\mathcal{U}_2[0,T]$, $J'_f(f(\cdot);\xi,u_1(\cdot),u_2(\cdot))$ is convex (resp., strictly convex) in $f(\cdot)\in L_{\mathbb{F}}^2(0,T;\mathbb{R}^{n})$ if and only if $J_1'(h(\cdot))$ is positive semidefinite (resp., positive definite), where
$$
J_1'(h(\cdot)):=\mathbb{E}\left\lbrace \int_0^T\left[-\left\langle Qz'_1,z'_1\right\rangle +\frac{\alpha}{2}\left\langle R_0h,h\right\rangle\right] \mathrm{d}t-\left\langle Gz'_1(T),z'_1(T)\right\rangle \right\rbrace,
$$
and $z'_1(\cdot)$ satisfies:
\begin{equation}\label{auxiliary control P1}
\left\{\begin{aligned}
	\mathrm{d}z'_1(t)&=[Az'_1+h]\mathrm{d}t+Cz'_1\mathrm{d}W ,\\
	z'_1(0)&=0.
\end{aligned}\right.
\end{equation}
\end{mylem}
	
\begin{proof}
For any $\xi\in\mathbb{R}^n$, $(u_1(\cdot),u_2(\cdot)\in \mathcal{U}_1[0,T] \times\mathcal{U}_2[0,T]$, let $x_1(\cdot)$, $x_2(\cdot)$ be the states of (\ref{state}) corresponding $g_1(\cdot)$, $g_2(\cdot)$, respectively. Taking any $\lambda_1\in[0,1]$ and denoting $\lambda_2:=1-\lambda_1$, we get
\begin{equation*}
\begin{aligned}
	&\lambda_1J'_f(g_1(\cdot))+\lambda_2J'_f(g_2(\cdot))-J'_f(\lambda_1g_1(\cdot)+\lambda_2g_2(\cdot))  \\
	&=\lambda_1\lambda_2\mathbb{E}\left\lbrace \int_0^T\left[-\left\langle Q(x_1-x_2),(x_1-x_2)\right\rangle +\frac{\alpha}{2}\left\langle R_0(g_1-g_2),(g_1-g_2)\right\rangle\right] dt-\left| x_1-x_2\right|_G^2 \right\rbrace.
\end{aligned}
\end{equation*}
Denote $h:=g_1-g_2$, $z'_1:=x_1-x_2$. Therefore, $z'_1(\cdot)$ is deterministic and satisfies (\ref{auxiliary control P1}). Hence
$$
\lambda_1J'_f(g_1(\cdot))+\lambda_2J'_f(g_2(\cdot))-J'_f(\lambda_1g_1(\cdot)+\lambda_2g_2(\cdot)) =\lambda_1\lambda_2J_1'(h(\cdot)),
$$
and the lemma follows.
\end{proof}

\begin{Remark}
When the attenuation index $\alpha$ or $R_0(\cdot)$ is sufficiently large, the convexity of $J'_f(f(\cdot);\xi,u_1(\cdot),u_2(\cdot))$ can usually be ensured.
\end{Remark}
	
We need the following assumptions.

{\bf (A3)} $R_0(\cdot)\gg0$.
	
{\bf (A4)} The map $h\mapsto J_1'(h)$ is uniformly positive definite.
	
\begin{mypro}\label{prop-3.1}
Let (A1)-(A4) hold. Let $\xi \in\mathbb{R}^n$, $(u_1(\cdot),u_2(\cdot))\in \mathcal{U}_1[0,T] \times\mathcal{U}_2[0,T]$ be given. Then, Problem (\textbf{LQ-F1}) is (uniquely) solvable if and only if there exists a (unique) triple $(\bar{x}(\cdot),\bar{p}(\cdot),\bar{\beta}(\cdot))$ satisfying the FBSDE:
\begin{equation}\label{robust-1}
\left\{\begin{aligned}
	\mathrm{d}\bar{x}(t)&=\bigg[A\bar{x}+B_1u_1+B_2u_2-\frac{2}{\alpha}R_0^{-1}\bar{p}\bigg]\mathrm{d}t+\big[C\bar{x}+D_1u_1+D_2u_2+\sigma\big]\mathrm{d}W, \\
	\mathrm{d}\bar{p}(t)&=\big[-A^\top\bar{p}-C^\top\bar{\beta}+Q\bar{x}\big]\mathrm{d}t+\bar{\beta}(t)\mathrm{d}W, \\
	\bar{x}(0)&=\xi, \quad \bar{p}(T)=-G\bar{x}(T).
\end{aligned}\right.
\end{equation}
Moreover, the optimal maximizer $\bar{f}(\cdot)$ is given by
\begin{equation}\label{wo-1}
	\bar{f}(t)=-\frac{2}{\alpha}R_0^{-1}(t)\bar{p}(t) , \quad  t\in[0,T].
\end{equation}
\end{mypro}
	
\begin{proof}
If Problem \textbf{(LQ-F1)} is solvable with an optimal maximizer $\bar{f}(\cdot)$, then by applying the well-known stochastic maximum principle, $\bar{p}(\cdot)$ is indeed the adjoint process. Thus the worst case of the disturbance is based on above Hamiltonian system (\ref{robust-1}).

Conversely, if (\ref{robust-1}) admits a solution $(\bar{x}(\cdot),\bar{p}(\cdot),\bar{\beta}(\cdot))$, then defining $\bar{f}(\cdot)$ by (\ref{wo-1}), for any $\lambda\in\mathbb{R}$, $f(\cdot)=\bar{f}(\cdot)+\lambda\delta f(\cdot)\in L_{\mathbb{F}}^2(0,T;\mathbb{R}^{n})$. Let the resulting solution be denoted by $(\bar{x}(\cdot)+\lambda\delta x(\cdot),\bar{p}(\cdot)+\lambda\delta p(\cdot),\bar{\beta}(\cdot)+\lambda\delta\beta(\cdot))$, where
\begin{equation*}
\left\{\begin{aligned}
	\mathrm{d}\delta x(t)&=[A\delta x+\delta f]\mathrm{d}t+C\delta x\mathrm{d}W, \\
	\mathrm{d}\delta p(t)&=\big[-A^\top\delta p-C^\top\delta\beta+Q\delta x\big]\mathrm{d}t+\delta\beta(t) \mathrm{d}W,   \\
	\delta x(0)&=0,  \quad  \delta p(T)=-G\delta x(T). \end{aligned}\right.
\end{equation*}
Thus,
\begin{equation}\label{diff-1}
\begin{aligned}
	&J'_f(f(\cdot))-J'_f(\bar{f}(\cdot))\\
    =&\ \lambda^2\biggl\{\mathbb{E}\int_0^T\left[-\left\langle Q\delta x,\delta x\right\rangle +\frac{\alpha}{2}\left\langle R_0\delta f,\delta f\right\rangle\right] \mathrm{d}t
    -\mathbb{E}\left\langle G\delta x(T),\delta x(T)\right\rangle \biggr\}  \\
	&+2\lambda\biggl\{\mathbb{E}\int_0^T\left[-\left\langle Q\bar{x},\delta x\right\rangle +\frac{\alpha}{2}\left\langle R_0\bar{f},\delta f\right\rangle\right] \mathrm{d}t
    -\mathbb{E}\left\langle G\bar{x}(T),\delta x(T)\right\rangle \biggr\}.
\end{aligned}
\end{equation}
By applying It\^o's formula to $\left\langle \bar{p}(\cdot),\delta x(\cdot)\right\rangle $, we obtain
\begin{equation}\label{d1}
	-\mathbb{E}\left\langle G\bar{x}(T),\delta x(T)\right\rangle=\mathbb{E}\int_0^T\big[\left\langle Q\bar{x},\delta x\right\rangle +\left\langle \bar{p},\delta f\right\rangle\big] \mathrm{d}t.
\end{equation}
Therefore, by (\ref{diff-1}) and (\ref{d1}),
\begin{equation*}
\begin{aligned}
	&J'_f(f(\cdot))-J'_f(\bar{f}(\cdot))\\
    =&\ \lambda^2\biggl\{\mathbb{E}\int_0^T\left[-\left\langle Q\delta x,\delta x\right\rangle +\frac{\alpha}{2}\left\langle R_0\delta f,\delta f\right\rangle\right] \mathrm{d}t
    -\mathbb{E}\left\langle G\delta x(T),\delta x(T)\right\rangle \biggr\}  \\
	&+2\lambda\mathbb{E}\int_0^T\left\langle \bar{p}+\frac{\alpha}{2}R_0\bar{f},\delta f\right\rangle \mathrm{d}t \\
	=&\ \lambda^2\biggl\{\mathbb{E}\int_0^T\left[-\left\langle Q\delta x,\delta x\right\rangle +\frac{\alpha}{2}\left\langle R_0\delta f,\delta f\right\rangle\right] \mathrm{d}t
-\mathbb{E}\left\langle G\delta x(T),\delta x(T)\right\rangle \biggr\}.
\end{aligned}
\end{equation*}
Due to (A4), hence
$$
J'_f(f(\cdot))\geqslant J'_f(\bar{f}(\cdot)).
$$
The uniqueness part follows easily.
\end{proof}
	
We discuss the solvability of the Hamiltonian system (\ref{robust-1}), which is a fully-coupled FBSDE. We introduce the following Riccati equation:
\begin{equation}\label{R-1}
\begin{aligned}
	\dot{P}_1+P_1A+A^\top P_1-\frac{2}{\alpha}P_1R_0^{-1}P_1+C^\top P_1C-Q=0, \quad P_1(T)=-G,
\end{aligned}
\end{equation}
and the following BSDE:
\begin{equation}\label{B-1}
\left\{\begin{aligned}
	\mathrm{d}\varphi_1(t)&=-\bigg[\bigg(A^\top-\frac{2}{\alpha}P_1R_0^{-1}\bigg)\varphi_1+P_1(B_1u_1+B_2u_2) \\
    &\qquad +C^\top P_1(D_1u_1+D_2u_2)+C^\top P_1\sigma+C^\top\theta_1\bigg]\mathrm{d}t +\theta_1(t)\mathrm{d}W, \\
	\varphi_1(T)&=0.
\end{aligned}\right.
\end{equation}
	
By Theorem 4.5 of Sun et al. \cite{Sun-Li-Yong16}, (A4) can ensure that Riccati equation (\ref{R-1}) admits a unique solution $P_1(\cdot)\in C([0,T];\mathbb{S}^n)$. Then, (\ref{B-1}) admits a solution $(\varphi_1(\cdot),\theta_1(\cdot))\in L_{\mathbb{F}}^2(\Omega;C([0,T];\mathbb{R}^n))\times L_{\mathbb{F}}^2(0,T;\mathbb{R}^n)$. Thus (\ref{robust-1}) is decoupled and solvable (see the Theorem 3.7 and Theorem 4.3 of Chapter 2 in Ma and Yong \cite{Ma-Yong99}). The wellposedness of (\ref{robust-1}) is obtained and further allows a closed-loop representation of the worst disturbance $\bar{f}(\cdot)$ being given
$$
\bar{f}(t)=-\frac{2}{\alpha}R_0^{-1}(t)\big[ P_1(t)\bar{x}(t)+\varphi_1(t)\big], \quad t\in[0,T].
$$
	
\subsection{Problem (\textbf{LQ-F2})}
	
By Proposition \ref{prop-3.1}, the worst disturbance $\bar{f}(\cdot)$ for given $\xi\in\mathbb{R}^n$, $(u_1(\cdot),u_2(\cdot))\in \mathcal{U}_1[0,T] \times\mathcal{U}_2[0,T]$ can be determined. This leaves the outer minimizing Problem (\textbf{LQ-F2}) for the follower:

Problem (\textbf{LQ-F2}): To find $\bar{u}_1(\cdot)\in \mathcal{U}_1[0,T]$, maximizing
\begin{equation}\label{J-wo}
\begin{aligned}
	&J^{wo}(u_1(\cdot);\xi,u_2(\cdot))\equiv\ J_f\big(\bar{f}(\cdot);\xi,u_1(\cdot),u_2(\cdot)\big) \\
	=&\ \mathbb{E}\int_0^T\Big[\left\langle Q(t)x(t),x(t)\right\rangle +\left\langle R_1(t)u_1(t),u_1(t)\right\rangle+\left\langle R_2(t)u_2(t),u_2(t)\right\rangle \\
	&\qquad\quad-\frac{2}{\alpha}\left\langle R_0^{-1}(t)p(t),p(t)\right\rangle\Big]\mathrm{d}t +\mathbb{E}\left\langle Gx(T),x(T)\right\rangle,
\end{aligned}
\end{equation}
subject to the state equation (\ref{robust-1}) which now becomes a FBSDE.

Problem \textbf{(LQ-F2)} is a LQ optimal control problem with FBSDE state. We need to discuss the convexity of (\ref{J-wo}) firstly.
\begin{mypro}
Let (A1)-(A4) hold. For any $(\xi,u_2(\cdot))\in \mathbb{R}^n\times\mathcal{U}_2[0,T]$, $J^{wo}(u_1(\cdot);\xi,u_2(\cdot))$ is convex (resp., strictly convex) in $u_1(\cdot)\in \mathcal{U}_1[0,T]$ if and only if $J_1{''}(v_1(\cdot))$ is positive semidefinite (resp., positive definite), where
\begin{equation*}
\begin{aligned}
	J_1^{''}(v_1(\cdot)):=\mathbb{E}\biggl\{\int_0^T\Big[\left\langle Qz_1{''},z_1{''}\right\rangle +\left\langle R_1v_1,v_1\right\rangle-\frac{2}{\alpha}\left\langle R_0^{-1}p_1,p_1\right\rangle\Big]\mathrm{d}t+\left\langle Gz_1{''}(T),z_1{''}(T)\right\rangle \biggr\},
\end{aligned}
\end{equation*}
and $(z_1^{''}(\cdot),p_1(\cdot))$ satisfies the FBSDE
\begin{equation}\label{auxiliary control P2}
\left\{\begin{aligned}
	\mathrm{d}z_1^{''}(t)&=\Big[Az_1^{''}+B_1v_1-\frac{2}{\alpha}R_0^{-1}p_1\Big]\mathrm{d}t+[Cz_1^{''}+D_1v_1]\mathrm{d}W, \\
	\mathrm{d}p_1(t)&=\big[-A^\top p_1-C^\top\beta_1+Qz_1^{''}\big]\mathrm{d}t+\beta_1(t)\mathrm{d}W,   \\
	z_1^{''}(0)&=0,  \quad  p_1(T)=-Gz_1^{''}(T).
\end{aligned}\right.
\end{equation}
\end{mypro}
	
For our further existence analysis, we introduce the following Riccati equation:
\begin{equation}\label{riccati}
\left\{\begin{aligned}
	&\dot{P}+PA+A^\top P+C^\top PC+Q\\
    &\ -(PB_1+C^\top PD_1)(R_1+D_1^\top PD_1)^{-1}(B_1^\top P+D_1^\top PC)=0, \\
    &P(T)=G,\quad \widetilde{R}_1:=R_1+D_1^\top PD_1\gg0.
\end{aligned}\right.
\end{equation}

Let us introduce the following assumptions, which will be used later.

{\bf(A5)} (\ref{riccati}) admits a unique strongly regular solution $P\in C([0,T];\mathbb{S}^n)$.
	
{\bf(A6)} The map $v_1\mapsto J_1^{''}(v_1)$ is positive definite.
	
\begin{mypro}\label{prop-3.3}
Let (A1)-(A4) hold and $(\xi,u_2(\cdot))\in \mathbb{R}^n\times\mathcal{U}_2[0,T]$ be fixed. Supposed for any $\delta u_1(\cdot)\in \mathcal{U}_1[0,T]$, the unique adapted solution  $(\delta x(\cdot),\delta p(\cdot),\delta \beta(\cdot),\delta y(\cdot),\delta z(\cdot),\delta q(\cdot))$ of the FBSDEs
\begin{equation}
\left\{\begin{aligned}
	\mathrm{d}\delta x(t)&=\Big[A\delta x+B_1\delta u_1-\frac{2}{\alpha}R_0^{-1}\delta p\Big]\mathrm{d}t+[C\delta x+D_1\delta u_1]\mathrm{d}W, \\
	\mathrm{d}\delta p(t)&=\big[-A^\top\delta p-C^\top\delta\beta+Q\delta x\big]\mathrm{d}t+\delta\beta(t) \mathrm{d}W, \\
	\mathrm{d}\delta y(t)&=\big[-A^\top\delta y-C^\top\delta z+Q\delta x-Q\delta q\big]\mathrm{d}t+\delta z(t)\mathrm{d}W,  \\
	\mathrm{d}\delta q(t)&=\Big[A\delta q+\frac{2}{\alpha}R_0^{-1}\delta y-\frac{2}{\alpha}R_0^{-1}\delta p\Big]\mathrm{d}t+C\delta q\mathrm{d}W, \\
	\delta x(0)&=0, \quad \delta p(T)=-G\delta x(T), \quad \delta y(T)=-G\delta x(T)+G\delta q(T), \quad \delta q(0)=0, \end{aligned}\right.
\end{equation}
satisfies
\begin{equation}\label{convexity condition-1}
	\mathbb{E}\int_0^T\big\langle R_1\delta u_1-B_1^\top\delta y-D_1^\top\delta z,\delta u_1\big\rangle dt\geqslant0.
\end{equation}
Then Problem \textbf{(LQ-F2)} is solvable with $(x(\cdot),p(\cdot),\beta(\cdot),\bar{u}_1(\cdot))$ being an optimal 4-tuple if and only if $(x(\cdot),p(\cdot),\beta(\cdot))$ is the adapted solution to (\ref{robust-1}) corresponding to the triple $(\xi,\bar{u}_1(\cdot),u_2(\cdot))$ and the FBSDE
\begin{equation}\label{adjoint equation-1}
\left\{\begin{aligned}
	\mathrm{d}y(t)&=-\big[A^\top y+C^\top z+Qq-Qx\big]\mathrm{d}t+z(t)\mathrm{d}W,  \\
	\mathrm{d}q(t)&=\Big[Aq+\frac{2}{\alpha}R_0^{-1}y-\frac{2}{\alpha}R_0^{-1}p\Big]\mathrm{d}t+Cq\mathrm{d}W, \\
	y(T)&=Gq(T)-Gx(T), \quad q(0)=0,
\end{aligned}\right.
\end{equation}
admits a unique adapted solution $(y(\cdot),z(\cdot),q(\cdot))$ such that
\begin{equation}\label{bar u_1}
	-R_1(t)\bar{u}_1(t)+B_1^\top(t)y(t)+D_1^\top(t)z(t)=0,\quad t\in[0,T].
\end{equation}
\end{mypro}

\begin{proof}
The proof is trivial and omitted.
\end{proof}
	
\begin{Remark}
We note that when (A6) holds, (\ref{convexity condition-1}) holds automatically.
\end{Remark}
	
\begin{mypro}\label{prop-3.4}
Let (A1)-(A4) hold. Then for any $(\xi,u_2(\cdot))\in \mathbb{R}^n\times\mathcal{U}_2[0,T]$, Problem \textbf{(LQ-F2)} admits a unique optimal control $\bar{u}_1(\cdot)\in\mathcal{U}_1[0,T]$, which is given by
\begin{equation}\label{bar u_1'}
	\bar{u}_1(t)=\widetilde{R}_1^{-1}\big[B_1^\top\bar{y}(t)+D_1^\top\bar{z}(t)-\widetilde{R}_1\bar{x}(t)-D_1^\top PD_2u_2(t)-D_1^\top P\sigma\big],
\end{equation}
where $(\bar{x}(\cdot),\bar{y}(\cdot),\bar{z}(\cdot))$ is solved by the FBSDEs:
\begin{equation}
\left\{\begin{aligned}
	d\bar{x}=&\biggl\{\Big[A-B_1\widetilde{R}_1^{-1}\big(B_1^\top P+D_1^\top PC\big)\Big]\bar{x}+B_1\widetilde{R}_1^{-1}B_1^\top \bar{y} +B_1\widetilde{R}_1^{-1}D_1^\top \bar{z}-\frac{2}{\alpha}R_0^{-1}\bar{p}\\
	&\quad+\Big(B_2-B_1-\frac{2}{\alpha}R_0^{-1}\bar{p}^{-1}D_1^\top PD_2\Big)u_2-B_1\widetilde{R}_1^{-1})D_1^\top P\sigma \biggr\}dt\\
	&\ +\biggl\{\Big[C-D_1\widetilde{R}_1^{-1}\big(B_1^\top P+D_1^\top PC\big)\Big]\bar{x}+D_1\widetilde{R}_1^{-1}B_1^\top \bar{y}+D_1\widetilde{R}_1^{-1}D_1^\top \bar{z} \\
	&\quad+\big(D_2-D_1\widetilde{R}_1^{-1}D_1^\top PD_2\big)u_2+\big(I-D_1\widetilde{R}_1^{-1}D_1^\top P\big)\sigma \biggr\}dW,\\
	d\bar{p}=&-\big[A^\top\bar{p}+C^\top\bar{\beta}-Q\bar{x}\big]dt+\bar{\beta}dW,\\
	d\bar{y}=&-\biggl\{\Big[A-B_1\widetilde{R}_1^{-1}\big(B_1^\top P+D_1^\top PC\big)\Big]^\top\bar{y}+\Big[C-D_1\widetilde{R}_1^{-1}\big(B_1^\top P+D_1^\top PC\big)\Big]^\top\bar{z} \\
	&\qquad+Q\bar{q}+\Big[\big(PB_1+C^\top PD_1\big)\widetilde{R}_1^{-1}D_1^\top PD_2-PB_2-C^\top PD_2\Big]u_2 \\
	&\qquad+\frac{2}{\alpha}PR_0^{-1}\bar{p}-C^\top P\sigma\biggr\}dt+\bar{z}dW,\\
	d\bar{q}=&\biggl\{A\bar{q}+\frac{2}{\alpha}R_0^{-1}\bar{y}-\frac{2}{\alpha}R_0^{-1}\bar{p}-\frac{2}{\alpha}R_0^{-1}P\bar{x}\biggr\}dt+C\bar{q}dW,\\
	\bar{x}(0)&=\xi,\quad\bar{p}(T)=-G\bar{x}(T),\quad\bar{y}(T)=-G\bar{x}(T)+G\bar{q}(T),\quad\bar{q}(0)=0.
\end{aligned}\right.
\end{equation}
\end{mypro}
	
Hence, the above result characterizes the Hamiltonian system of Problem \textbf{(LQ-F2)} which is a coupled system of high-dimensional FBSDEs:
\begin{equation}\label{state-1}
\left\{\begin{aligned}
	\mathrm{d}\bar{X}(t)&=\big[\hat{A}_1\bar{X}+\hat{B}_1\bar{Y}+\hat{B}_3\bar{Z}+\hat{B}_2u_2+\hat{b}\big]\mathrm{d}t +\big[\hat{C}\bar{X}+\hat{D}_1\bar{Y}+\hat{D}_3\bar{Z}+\hat{D}_2u_2+\hat{\sigma}\big]\mathrm{d}W, \\
	\mathrm{d}\bar{Y}(t)&=\big[-\hat{A}_2^\top\bar{Y}-\hat{C}^\top\bar{Z}+\hat{Q}\bar{X}+\hat{F}u_2+\hat{v}\big]\mathrm{d}t +\bar{Z}(t)\mathrm{d}W, \\
	\bar{X}(0)&=\hat{\xi}, \quad \bar{Y}(T)=\hat{G}\bar{X}(T),
\end{aligned}\right.
\end{equation}
where by (\ref{robust-1}), (\ref{adjoint equation-1}) and (\ref{bar u_1}), we have denoted
\begin{equation*}
\left\{\begin{aligned}
	&\bar{X}:=(\bar{x}^\top,\bar{q}^\top)^\top, \quad \bar{Y}:=(\bar{y}^\top,\bar{p}^\top)^\top, \quad \bar{Z}:=(\bar{z}^\top,\bar{\beta}^\top)^\top, \\
	&\hat{A}_1:=
	\begin{pmatrix}
		A-B_1\widetilde{R}_1^{-1}\big(B_1^\top P+D_1^\top PC\big) & 0 \\
		-\frac{2}{\alpha}R_0^{-1}P & A
	\end{pmatrix}, \\
	&\hat{A}_2:=
	\begin{pmatrix}
		A-B_1\widetilde{R}_1^{-1}\big(B_1^\top P+D_1^\top PC\big) & 0 \\
		\frac{2}{\alpha}R_0^{-1}P & A
	\end{pmatrix}, \\
	 &\hat{C}:=
	 \begin{pmatrix}
	 	C-D_1\widetilde{R}_1^{-1}\big(B_1^\top P+D_1^\top PC\big) & 0 \\
	 	0 & C
	 \end{pmatrix}, \\ &\hat{B}_1:=
	 \begin{pmatrix}
	 	B_1\widetilde{R}_1^{-1}B_1^\top & -\frac{2}{\alpha}R_0^{-1} \\
	 	\frac{2}{\alpha}R_0^{-1} & -\frac{2}{\alpha}R_0^{-1}
	 \end{pmatrix}, \quad
	\hat{B}_3:=
	\begin{pmatrix}
		B_1\widetilde{R}_1^{-1}D_1^\top & 0 \\
		0 & 0
	\end{pmatrix}, \\
\end{aligned}\right.
\end{equation*}
and
\begin{equation*}
\left\{\begin{aligned}
	&\hat{B}_2:=
	\begin{pmatrix}
		B_2-B_1\widetilde{R}_1^{-1}D_1^\top PD_2 \\0
	\end{pmatrix}, \quad
	\hat{D}_1:=
	\begin{pmatrix}
			D_1\widetilde{R}_1^{-1}B_1^\top & 0 \\
			0 & 0
	\end{pmatrix}, \\
	&\hat{D}_3:=
	\begin{pmatrix}
		D_1\widetilde{R}_1^{-1}D_1^\top & 0 \\
		0 & 0
	\end{pmatrix}, \quad
	\hat{D}_2:=
	\begin{pmatrix}
		D_2-D_1\widetilde{R}_1^{-1}D_1^\top PD_2 \\0
	\end{pmatrix}, \\
    &\hat{b}:=
	\begin{pmatrix}
		-B_1\widetilde{R}_1^{-1}D_1^\top P\sigma \\ 0
	\end{pmatrix}, \quad
	\hat{\sigma}:=
	\begin{pmatrix}
		\big(I-D_1\widetilde{R}_1^{-1}D_1^\top P\big)\sigma \\ 0
	\end{pmatrix}, \\
	&\hat{v}:=
	\begin{pmatrix}
		C^\top P\sigma \\ 0
	\end{pmatrix}, \quad
	\hat{\xi}:=
	\begin{pmatrix}
		\xi \\ 0
	\end{pmatrix},\\
	&\hat{F}:=
	\begin{pmatrix}
		-\big(PB_1+C^\top PD_1\big)\widetilde{R}_1^{-1}D_1^\top PD_2+PB_2+C^\top PD_2 \\ 0
	\end{pmatrix},\\
	&\hat{Q}:=
	\begin{pmatrix}
		0 & -Q \\
		Q & 0
	\end{pmatrix}, \quad
	\hat{G}:=
	\begin{pmatrix}
		-G & G \\
		-G & 0
	\end{pmatrix}.
\end{aligned}\right.
\end{equation*}

\begin{proof}
	The proof is trivial and omitted.
\end{proof}

\begin{mypro}\label{prop-3.5}
Let (A1)-(A2), (A5) hold and $u_2(\cdot)\in\mathcal{U}_2[0,T]$ be fixed. Supposed the Riccati equation
\begin{equation}\label{R-2}
\left\{\begin{aligned}
	&\dot{P}_2+P_2\hat{A}_1+\hat{A}_2^\top P_2+P_2\hat{B}_1P_2-\hat{Q}\\
    &\ +(\hat{C}^\top+P_2\hat{B}_3)(I-P_2\hat{D}_3)^{-1}(P_2\hat{C}+P_2\hat{D}_1P_2)=0, \\
	&P_2(T)=\hat{G},
\end{aligned}\right.
\end{equation}
and the following BSDE:
\begin{equation}\label{B-2}
\left\{\begin{aligned}
	\mathrm{d}\varphi_2(t)=&-\biggl\{\left( \hat{A}_2^\top+P_2\hat{B}_1+(\hat{C}^\top+P_2\hat{B}_3)(I-P_2\hat{D}_3)^{-1}P_2\hat{D}_1\right) \varphi_2 \\
	&\qquad +(\hat{C}^\top+P_2\hat{B}_3)(I-P_2\hat{D}_3)^{-1}(P_2\hat{D}_2u_2+P_2\hat{\sigma}+\theta_2)\\
    &\qquad +(P_2\hat{B}_2-\hat{F})u_2+P_2\hat{b}-\hat{v} \biggr\}\mathrm{d}t+\theta_2(t)\mathrm{d}W, \\
	\varphi_2(T)=&\ 0,
\end{aligned}\right.
\end{equation}
admit solutions $P_2(\cdot)\in C([0,T];\mathbb{R}^{2n\times2n})$, $(\varphi_2(\cdot),\theta_2(\cdot))\in L_{\mathbb{F}}^2(\Omega;C([0,T];\mathbb{R}^{2n}))\times L_{\mathbb{F}}^2(0,T;\mathbb{R}^{2n})$, respectively. Then, the wellposedness of (\ref{state-1}) follows.
\end{mypro}
		
\begin{proof}
The main idea of the proof is applying It\^o's formula to $P_2(\cdot)\bar{X}(\cdot)+\varphi_2(\cdot)$ and comparing the coefficients with (\ref{state-1}), we omit it here.
\end{proof}
		
\section{Robust Stackelberg strategy for the leader}
		
Given Proposition \ref{prop-3.4} and Proposition \ref{prop-3.5}, if Riccati equation (\ref{R-2}) is solvable, the best response of follower can be characterized for any given $(u_2(\cdot),\xi)\in\mathcal{U}_2[0,T]\times\mathbb{R}^n$. Now we turn to the optimal control problem from the viewpoint of the leader to have robust design. Because of unknown process $f_2(\cdot)$, we approach the SLQ problem from a robust optimation viewpoint where $f_2(\cdot)$ is treated as an adversarial (minimizing) player. Here, a soft-constraint for the disturbance is adapted in that the term $\frac{\gamma}{2}\langle\hat{R_0}f_2,f_2\rangle$ is included in the following
\begin{equation}\label{cost_L1}
\begin{aligned}
 	&\tilde{J}_f(f_2(\cdot);\xi,u_2(\cdot))\equiv J_f(f_2(\cdot);\xi,\bar{u}_1(\cdot),u_2(\cdot)) \\
	&=\mathbb{E}\bigg\{\int_0^T \Big[\langle Qx,x\rangle +\Big\langle R_1\widetilde{R}_1^{-1}\big[B_1^\top(t)\bar{y}(t)+D_1^\top(t)\bar{z}(t) \\
	&\qquad\quad-(B_1^\top P+D_1^\top PC)\bar{x}(t)-D_1^\top PD_2u_2(t)-D_1^\top P\sigma\big],\widetilde{R}_1^{-1}\big[B_1^\top(t)\bar{y}(t)\\
    &\qquad\quad+D_1^\top(t)\bar{z}(t)-(B_1^\top P+D_1^\top PC)\bar{x}(t)-D_1^\top PD_2u_2(t)-D_1^\top P\sigma\big]\Big\rangle \\
	&\qquad\quad +\langle R_2u_2,u_2\rangle+\frac{\gamma}{2}\langle \hat{R}_0f_2,f_2\rangle \Big] \mathrm{d}t+\langle Gx(T),x(T)\rangle \bigg\},
\end{aligned}
\end{equation}
while $f_2(\cdot)$ attempts to minimize the functional $\tilde{J}_f(f_2(\cdot);\xi,u_2(\cdot))$. Note that $\gamma$ is called the \textit{attenuation} parameter of soft-constraint (\cite{Huang-Huang17}, \cite{Huang-Wang-Wu22}).

Summarizing the above mentioned discussion, because of the informational uncertainty, the leader considers the worst-case analysis and confronts the following sup-inf problem:
		
Problem \textbf{(Sup-Inf)}:
$$
\sup_{u_2(\cdot)\in\,\mathcal{U}_2[0,T]}\inf_{f_2(\cdot)\in L_{\mathbb{F}}^2(0,T;\mathbb{R}^n)}\tilde{J}_f(f_2(\cdot);\xi,u_2(\cdot)),\quad \mbox{ for any }\xi\in\mathbb{R}^n.
$$
Given the above mentioned sup-inf formulation, we need first tackle the inner LQ minimizing problem for disturbance:
		
Problem \textbf{(LQ-L1)}: To find $\bar{f}_2(\cdot)\in L_{\mathbb{F}}^2(0,T;\mathbb{R}^{n})$ such that
$$
\tilde{J}_f(\bar{f}_2(\cdot);\xi,u_2(\cdot))=\inf_{f_2(\cdot)\in L_{\mathbb{F}}^2(0,T;\mathbb{R}^n)}\tilde{J}_f(f_2(\cdot);\xi,u_2(\cdot)),
$$
subject to $\xi \in\mathbb{R}^{n}$, $u_2(\cdot)\in\mathcal{U}_2[0,T]$.

Define a map $\bar{\alpha}_3:\mathbb{R}^{n}\times\mathcal{U}_2[0,T]\mapsto L_{\mathbb{F}}^2(0,T;\mathbb{R}^{n})$. We may denote the optimal minimizer $\bar{f}_2(\cdot)\equiv\bar{\alpha}_3\left[ \xi,u_2(\cdot)\right] (\cdot) $, which depends on the tuple $(\xi ,u_2(\cdot))$. Define
$$
\tilde{J}^{wo}(u_2(\cdot);\xi)=\tilde{J}_f(\bar{\alpha}_3\left[ \xi,u_2(\cdot)\right] (\cdot);\xi,u_2(\cdot)).
$$
Given the mapping $\bar{\alpha}_3$, the leader needs solve the outer SLQ problem:

Problem \textbf{(LQ-L2)}: To find $\bar{u}_2(\cdot)\in \mathcal{U}_2[0,T]$ such that
$$
\tilde{J}^{wo}(\bar{u}_2(\cdot);\xi)=\sup_{u_2(\cdot)\in\,\mathcal{U}_2[0,T]}\tilde{J}^{wo}(u_2(\cdot);\xi),\quad \mbox{subject to }\xi \in\mathbb{R}^n.
$$		
		
\subsection{Problem (\textbf{LQ-L1})}
		
We first deal with the inner minimizing problem, namely \textbf{(LQ-L1)}, the state is given by the following FBSDEs:
\begin{equation}\label{state L1}
\left\{\begin{aligned}
	dx&=\Big\{Ax-B_1\widetilde{R}_1^{-1}\big(B_1^\top P+D_1^\top PC\big)\bar{x}+B_1\widetilde{R}_1^{-1}B_1^\top\bar{y}+B_1\widetilde{R}_1^{-1}D_1^\top\bar{z} \\
	&\qquad+\big(B_2-B_1\widetilde{R}_1^{-1}D_1^\top PD_2\big)u_2+f_2+f_1-B_1\widetilde{R}_1^{-1}D_1^\top P\sigma\Big\}dt \\
	&\qquad +\Big\{Cx-D_1\widetilde{R}_1^{-1}\big(B_1^\top P+D_1^\top PC\big)\bar{x}+D_1\widetilde{R}_1^{-1}B_1^\top\bar{y}+D_1\widetilde{R}_1^{-1}D_1^\top\bar{z} \\
	&\qquad+\big(D_2-D_1\widetilde{R}_1^{-1}D_1^\top PD_2\big)u_2+\big(I-D_1\widetilde{R}_1^{-1}D_1^\top P\big)\sigma\Big\}dW, \\				
    d\bar{x}=&\Big\{\big[\big(A-B_1\widetilde{R}_1^{-1}\big(B_1^\top P+D_1^\top PC\big)\big]\bar{x}+B_1\widetilde{R}_1^{-1}B_1^\top \bar{y}+B_1\widetilde{R}_1^{-1}D_1^\top \bar{z} \\
    &\quad-\frac{2}{\alpha}R_0^{-1}\bar{p}+\big(B_2-B_1\widetilde{R}_1^{-1}D_1^\top PD_2\big)u_2-B_1\widetilde{R}_1^{-1})D_1^\top P\sigma \Big\}dt \\
    &\ +\Big\{\big[C-D_1\widetilde{R}_1^{-1}\big(B_1^\top P+D_1^\top PC\big)\big]\bar{x}+D_1\widetilde{R}_1^{-1}B_1^\top \bar{y}+D_1\widetilde{R}_1^{-1}D_1^\top \bar{z} \\
    &\quad+\big(D_2-D_1\widetilde{R}_1^{-1}D_1^\top PD_2\big)u_2+\big(I-D_1\widetilde{R}_1^{-1}D_1^\top P\big)\sigma \Big\}dW,\\
    d\bar{p}=&-\big[A^\top\bar{p}+C^\top\bar{\beta}-Q\bar{x}\big]dt+\bar{\beta}dW,\\
    d\bar{y}=&-\Big\{\big[A-B_1\widetilde{R}_1^{-1}\big(B_1^\top P+D_1^\top PC\big)\big]^\top\bar{y}+\big[C-D_1\widetilde{R}_1^{-1}\big(B_1^\top P+D_1^\top PC\big)\big]^\top\bar{z} \\
    &\qquad+Q\bar{q}+\big[\big(PB_1+C^\top PD_1\big)\widetilde{R}_1^{-1}D_1^\top PD_2-PB_2+C^\top PD_2\big]u_2 \\
    &\qquad+\frac{2}{\alpha}PR_0^{-1}\bar{p}-C^\top P\sigma\Big\}dt+\bar{z}dW,\\
    d\bar{q}=&\Big[A\bar{q}+\frac{2}{\alpha}R_0^{-1}\bar{y}-\frac{2}{\alpha}R_0^{-1}\bar{p}-\frac{2}{\alpha}R_0^{-1}P\bar{x}\Big]dt+C\bar{q}dW,\\
    x(0)&=\xi,\quad\bar{x}(0)=\xi,\quad\bar{p}(T)=-G\bar{x}(T),\quad\bar{y}(T)=-G\bar{x}(T)+G\bar{q}(T),\quad\bar{q}(0)=0,
\end{aligned}\right.
\end{equation}
and the cost functional is given by (\ref{cost_L1}). For any given  $(u_2(\cdot),\xi)\in\mathcal{U}_2[0,T]\times\mathbb{R}^n$, from Proposition \ref{prop-3.5}, the wellposedness of (\ref{state-1}) can be ensured, which admits a unique solution $(\bar{X}(\cdot),\bar{Y}(\cdot),\bar{Z}(\cdot))$.

To solve Problem \textbf{(Sup-Inf)}, we assume

{\bf(A7)} (\romannumeral1) The Riccati equation (\ref{R-2}) admits a solution $P_2(\cdot)$.\quad (\romannumeral2) $\hat{R}_0(\cdot)\gg0$.\\
		(\romannumeral3) The map $g\mapsto J'_2(g)$ is positive definite.
		
Note that under (A7) (\romannumeral1), Problem \textbf{(LQ-L1)} can be rewritten as an equivalent problem:
		
Problem \textbf{(LQ-L1a)}: To find $\bar{f}_2(\cdot)\in L_{\mathbb{F}}^2(0,T;\mathbb{R}^n)$ such that
$$
\tilde{J}'_f(\bar{f}_2(\cdot);\xi,u_2(\cdot))=\inf_{f_2(\cdot)\in L_{\mathbb{F}}^2(0,T;\mathbb{R}^n)}\tilde{J}'_f(f_2(\cdot);\xi,u_2(\cdot)),
$$
subject to $\xi \in\mathbb{R}^n$, $u_2(\cdot)\in\mathcal{U}_2[0,T]$, where
\begin{equation}\label{cost_L1'}
	\tilde{J}'_f(f_2(\cdot);\xi,u_2(\cdot)):=\mathbb{E}\bigg\{\int_0^T \Big[\langle Qx,x\rangle +\frac{\gamma}{2}\langle \hat{R}_0f_2,f_2\rangle \Big] \mathrm{d}t +\langle Gx(T),x(T)\rangle \bigg\}.
\end{equation}

For further proofs, we need to discuss the convexity of (\ref{cost_L1'}).
\begin{mypro}\label{prop-4.1}
Let (A1)-(A6) hold. For any $(u_2(\cdot),\xi)\in\mathcal{U}_2[0,T]\times\mathbb{R}^n$, $\tilde{J}_f'(f_2(\cdot);\xi,u_2(\cdot))$ is convex (resp., strictly convex) in $f_2(\cdot)\in L_{\mathbb{F}}^2(0,T;\mathbb{R}^{n})$ if and only if $J'_2(g(\cdot))$ is positive semidefinite (resp., positive definite), where
$$
\tilde{J}'_2(g(\cdot)):=\mathbb{E}\bigg\{\int_0^T \Big[\langle Qz'_2,z'_2\rangle +\frac{\gamma}{2}\langle \hat{R}_0g,g\rangle \Big] \mathrm{d}t +\langle Gz'_2(T),z'_2(T)\rangle \bigg\},
$$
and $z'_2(\cdot)$ satisfies
\begin{equation}
\left\{\begin{aligned}
	\mathrm{d}z'_2(t)&=[Az'_2+g]\mathrm{d}t+Cz'_2\mathrm{d}W, \\
	z'_2(0)&=0.
\end{aligned}\right.
\end{equation}
\end{mypro}
		
\begin{proof}
Similar to the proof of Lemma \ref{lemma3.1}, we omit it.
\end{proof}
		
For notational simplicity, we set
\begin{equation*}
\left\{\begin{aligned}
	&\hat{X}:=(x^\top,\bar{x}^\top,\bar{q}^\top)^\top, \quad \hat{Y}:=(\bar{y}^\top,\bar{p}^\top)^\top, \quad \hat{Z}:=(\bar{z}^\top,\bar{\beta}^\top)^\top,\\
    &\check{I}:=(I_n^\top,0,0)^\top,\quad\check{\xi}:=(\xi^\top,\xi^\top,0)^\top, \\
	&\check{A}:=
	\begin{pmatrix}
		A & -B_1\widetilde{R}_1^{-1}\big(B_1^\top P+D_1^\top PC\big) & 0 \\
		0 & A-B_1\widetilde{R}_1^{-1}\big(B_1^\top P+D_1^\top PC\big) & 0 \\
		0 & -\frac{2}{\alpha}R_0^{-1}P & A
	\end{pmatrix}, \\
	&\check{C}:=
	\begin{pmatrix}
	 	C & -D_1\widetilde{R}_1^{-1}\big(B_1^\top P+D_1^\top PC\big) & 0 \\
	 	0 & C-D_1\widetilde{R}_1^{-1}\big(B_1^\top P+D_1^\top PC\big) & 0 \\
	 	0 & 0 & C
	\end{pmatrix}, \\
	&\check{B}_1:=
	\begin{pmatrix}
		B_1\widetilde{R}_1^{-1}B_1^\top& 0\\
		B_1\widetilde{R}_1^{-1}B_1^\top& -\frac{2}{\alpha}R_0^{-1}\\
		\frac{2}{\alpha}R_0^{-1}&-\frac{2}{\alpha}R_0^{-1}
	\end{pmatrix}, \quad
	 \check{B}_3:=
	 \begin{pmatrix}
		B_1\widetilde{R}_1^{-1}D_1^\top& 0\\
		B_1\widetilde{R}_1^{-1}D_1^\top& 0\\
		0&0
	 \end{pmatrix}, \\
	 &\check{D}_1:=
	 \begin{pmatrix}
		D_1\widetilde{R}_1^{-1}B_1^\top& 0\\
		D_1\widetilde{R}_1^{-1}B_1^\top& 0\\
		0&0
	 \end{pmatrix}, \quad
	 \check{D}_3:=
	 \begin{pmatrix}
		D_1\widetilde{R}_1^{-1}D_1^\top& 0\\
		D_1\widetilde{R}_1^{-1}D_1^\top& 0\\
		0&0
	\end{pmatrix}, \\
	&\check{B}_2:=
	\begin{pmatrix}
		B_2-B_1\widetilde{R}_1^{-1}D_1^\top PD_2\\
		B_2-B_1\widetilde{R}_1^{-1}D_1^\top PD_2\\0
	\end{pmatrix}, \quad
	\check{D}_2:=
	\begin{pmatrix}
		D_2-D_1\widetilde{R}_1^{-1}D_1^\top PD_2\\
		D_2-D_1\widetilde{R}_1^{-1}D_1^\top PD_2\\0
	\end{pmatrix}, \\
	&\check{F_1}:=
	\begin{pmatrix}
		f_1-B_1\widetilde{R}_1^{-1}D_1^\top P\sigma\\
		-B_1\widetilde{R}_1^{-1}D_1^\top P\sigma\\0
	\end{pmatrix}, \quad
	\check{\sigma}:=
	\begin{pmatrix}
		\big(I-D_1\widetilde{R}_1^{-1}D_1^\top P\big)\sigma\\
		\big(I-D_1\widetilde{R}_1^{-1}D_1^\top P\big)\sigma\\0
	\end{pmatrix}, \\
	&\check{Q}:=
	\begin{pmatrix}
		0 & 0 & -Q \\
		0 & Q & 0
	\end{pmatrix}, \quad
	\check{G}:=
	\begin{pmatrix}
		0 & -G & G \\
		0 & -G & 0
	\end{pmatrix}, \\
	&\bar{Q}:=
	\begin{pmatrix}
		Q & 0 & 0 \\
		0 & 0 & 0 \\
		0 & 0 & 0
	\end{pmatrix}, \quad
	\bar{G}:=
	\begin{pmatrix}
		G & 0 & 0 \\
		0 & 0 & 0 \\
		0 & 0 & 0
	\end{pmatrix}.
\end{aligned}\right.
\end{equation*}
Then (\ref{state L1}), together with (\ref{cost_L1'}), is equivalent to the following FBSDE:
\begin{equation}\label{state L1'e}
\left\{\begin{aligned}					
    \mathrm{d}\hat{X}(t)&=\big[\check{A}\hat{X}+\check{B}_1\hat{Y}+\check{B}_3\hat{Z}+\check{B}_2u_2+\check{F}_1+\check{I}f_2\big]\mathrm{d}t\\
    &\quad +\big[\check{C}\hat{X}+\check{D}_1\hat{Y}+\check{D}_3\hat{Z}+\check{D}_2u_2+\check{\sigma}\big]\mathrm{d}W, \\
	\mathrm{d}\hat{Y}(t)&=\big[-\hat{A}^\top\hat{Y}-\hat{C}^\top\hat{Z}+\check{Q}\hat{X}+\hat{F}u_2+\hat{v}\big]\mathrm{d}t+\hat{Z}(t)\mathrm{d}W, \\
	\hat{X}(0)&=\check{\xi}, \quad \hat{Y}(T)=\check{G}\hat{X}(T),
\end{aligned}\right.
\end{equation}
and	
\begin{equation}\label{cost_L1'e}
\begin{aligned}
	\tilde{J}'_f(f_2(\cdot);\xi,u_2(\cdot))=\mathbb{E}\bigg\{\int_0^T \Big[\langle \bar{Q}\hat{X},\hat{X}\rangle +\frac{\gamma}{2}\langle \hat{R}_0f_2,f_2\rangle \Big] \mathrm{d}t
    +\langle \bar{G}\hat{X}(T),\hat{X}(T)\rangle \bigg\}.
\end{aligned}
\end{equation}
		
\begin{mythm}\label{thm-4.1}
Suppose that (A1)-(A7) hold. For any $(u_2(\cdot),\xi)\in\mathcal{U}_2[0,T]\times\mathbb{R}^n$, then Problem \textbf{(LQ-L1)} has a (unique) minimizer $\bar{f}_2(\cdot)$ if and only if (\ref{state L1'e}) admits a solution corresponding to $\bar{f}_2(\cdot)$ and the following adjoint equation
\begin{equation}\label{adjoint equation-2}
\left\{\begin{aligned}
	\mathrm{d}\hat{q}(t)&=-\big[\check{A}^\top\hat{q}+\check{C}^\top\hat{k}+\check{Q}^\top\hat{p}-\bar{Q}\hat{X}\big]\mathrm{d}t+\hat{k}(t)\mathrm{d}W, \\
	\mathrm{d}\hat{p}(t)&=\big[\hat{A}\hat{p}-\check{B}_1^\top\hat{q}-\check{D}_1^\top\hat{k}\big]\mathrm{d}t +\big[\hat{C}\hat{p}-\check{B}_3^\top\hat{q}-\check{D}_3^\top\hat{k}\big]\mathrm{d}W, \\
	\hat{q}(T)&=-\check{G}^\top\hat{p}(T)-\bar{G}\hat{X}(T), \quad \hat{p}(0)=0,
\end{aligned}\right.
\end{equation}
admits a (unique) adapted solution $(\hat{p}(\cdot),\hat{q}(\cdot),\hat{k}(\cdot))$, where
\begin{equation*}
	\hat{p}:=\begin{pmatrix}
		\tilde{\bar{y}}\\
		\tilde{\bar{p}}
	\end{pmatrix}, \quad
	\hat{q}:=\begin{pmatrix}
		\tilde{x}\\
		\tilde{\bar{x}}\\
		\tilde{\bar{q}}
	\end{pmatrix}, \quad
	\hat{k}:=\begin{pmatrix}
		\zeta\\
		\upsilon\\
		\varsigma
	\end{pmatrix}.
\end{equation*}	
Here, the superscript ``$\,\,\tilde{\cdot}\,$" represents the adjoint part of the corresponding processes. Moreover, the optimal minimizer is given by
\begin{equation}\label{wo-2}
	\bar{f}_2(t)=\frac{2}{\gamma}\hat{R}_0^{-1}(t)\tilde{x}(t), \quad t\in[0,T].
\end{equation}
\end{mythm}
		
\begin{proof}
Similar to the proof of Proposition \ref{prop-3.1}, we omit it.
\end{proof}
		
Let us set
\begin{equation*}
\left\{\begin{aligned}
	&\mathbb{X}:=(\hat{X}^\top,\hat{p}^\top)^\top, \quad \mathbb{Y}:=(\hat{q}^\top,\hat{Y}^\top)^\top, \quad \mathbb{Z}:=(\hat{k}^\top,\hat{Z}^\top)^\top, \\
	&\mathbb{A}:=
	\begin{pmatrix}
		\check{A}& 0\\
		0& \hat{A}
	\end{pmatrix}, \quad \mathbb{C}:=\begin{pmatrix}
		\check{C}& 0\\
		0& \hat{C}
	\end{pmatrix}, \quad
    \mathbb{B}_2:=(\check{B}_2^\top,0)^\top, \\
    &\mathbb{F}_1:=(\check{F}_1^\top,0)^\top, \quad
	\mathbb{F}_2:=(0,\hat{F}^\top)^\top,\quad
    \mathbb{D}_2:=(\check{D}_2^\top,0)^\top, \\
    &\Sigma:=(\check{\sigma}^\top,0)^\top, \quad
    \Xi:=(\check{\xi}^\top,0)^\top,\quad
    \Upsilon:=(0,\hat{v}^\top)^\top,\\
	&\mathbb{B}_1:=
	\begin{pmatrix}
		\frac{2}{\gamma}\check{I}\hat{R}_0^{-1}\check{I}^\top& \check{B}_1\\
		-\check{B}_1^\top&0
	\end{pmatrix}, \quad
	\mathbb{B}_3:=
	\begin{pmatrix}
	0 & \check{B}_3\\
	-\check{D}_1^\top & 0
	\end{pmatrix}, \quad
	\mathbb{D}_1:=
	\begin{pmatrix}
		0&\check{D}_1 \\
		-\check{B}_3^\top & 0
	\end{pmatrix}, \\
	&\mathbb{D}_3:=
	\begin{pmatrix}
		0&\check{D}_3 \\
		-\check{D}_3^\top& 0
	\end{pmatrix}, \quad
	\mathbb{Q}:=
	\begin{pmatrix}
		\bar{Q} &-\check{Q}^\top \\
		 \check{Q}& 0
	\end{pmatrix}, \quad
	\mathbb{G}:=
	\begin{pmatrix}
		-\bar{G}& -\check{G}^\top \\
		\check{G}  & 0
	\end{pmatrix}.
\end{aligned}\right.
\end{equation*}
Then, (\ref{state L1'e}), (\ref{adjoint equation-2}) and (\ref{wo-2}) are equivalent to the FBSDE:
\begin{equation}\label{state L2}
\left\{\begin{aligned}
	\mathrm{d}\mathbb{X}(t)&=\big[\mathbb{A}\mathbb{X}+\mathbb{B}_1\mathbb{Y}+\mathbb{B}_3\mathbb{Z}+\mathbb{B}_2u_2+\mathbb{F}_1\big]\mathrm{d}t
    +\big[\mathbb{C}\mathbb{X}+\mathbb{D}_1\mathbb{Y}+\mathbb{D}_3\mathbb{Z}+\mathbb{D}_2u_2+\Sigma\big]\mathrm{d}W, \\
	\mathrm{d}\mathbb{Y}(t)&=\big[-\mathbb{A}^\top\mathbb{Y}-\mathbb{C}^\top\mathbb{Z}+\mathbb{Q}\mathbb{X}+\mathbb{F}_2u_2+\Upsilon\big]\mathrm{d}t+\mathbb{Z}(t)\mathrm{d}W, \\
	\mathbb{X}(0)&=\Xi, \quad \mathbb{Y}(T)=\mathbb{G}\mathbb{X}(T).
\end{aligned}\right.
\end{equation}
		
We now use the idea of the four-step scheme (see, for example, \cite{Yong02}) to study the solvability of the above Hamiltonian system.
\begin{mypro}\label{prop-4.2}
Let (A1)-(A7) hold, if Riccati equation
\begin{equation}
\left\{\begin{aligned}\label{R-3}
	&\dot{P}_3+P_3\mathbb{A}+\mathbb{A}^\top P_3+P_3\mathbb{B}_1P_3-\mathbb{Q}+(\mathbb{C}^\top+P_3\mathbb{B}_3)(I-P_3\mathbb{D}_3)^{-1}(P_3\mathbb{C}+P_3\mathbb{D}_1P_3)=0, \\
    & P_3(T)=\mathbb{G},
\end{aligned}\right.
\end{equation}
admits a unique solution $P_3(\cdot)\in C([0,T];\mathbb{R}^{5n\times5n})$ over $[0,T]$, then the following BSDE admits a unique solution $(\varphi(\cdot),\Lambda(\cdot))\in L_{\mathbb{F}}^2(\Omega;C([0,T];\mathbb{R}^{5n}))\times L_{\mathbb{F}}^2(0,T;\mathbb{R}^{5n})$:
\begin{equation}
\left\{\begin{aligned}\label{B-3}
	\mathrm{d}\varphi(t)&=-\Big\{\big[\mathbb{A}^\top+P_3\mathbb{B}_1+(\mathbb{C}^\top+P_3\mathbb{B}_3)(I-P_3\mathbb{D}_3)^{-1}P_3\mathbb{D}_1\big]\varphi\\
    &\qquad +(\mathbb{C}^\top+P_3\mathbb{B}_3)(I-P_3\mathbb{D}_3)^{-1}\big[P_3(\mathbb{D}_2u_2+\Sigma)+\Lambda\big]\\
    &\qquad +(P_3\mathbb{B}_2-\mathbb{F}_2)u_2+P\mathbb{F}_1-\Upsilon\Big\}\mathrm{d}t+\Lambda(t) \mathrm{d}W, \\
	\varphi(T)&=0,
\end{aligned}\right.
\end{equation}
and the wellposedness of system (\ref{state L2}) is obtained.
\end{mypro}
		
\begin{proof}
Similar to the proof of Proposition \ref{prop-3.5}.
\end{proof}
		
\subsection{Problem (\textbf{LQ-L2})}
		
By Theorem \ref{thm-4.1}, Proposition \ref{prop-4.2},  the worst disturbance $\bar{f}_2(\cdot)$ for given $(\xi,u_2(\cdot))$ can be determined. The leader now faces the outer maximizing Problem \textbf{(LQ-L2)}. Precisely, we rewrite
\begin{equation}\label{J_L1-wo}
\begin{aligned}
	&\tilde{J}^{wo}(u_2(\cdot);\xi)\equiv \tilde{J}_f(\bar{f}_2(\cdot);\xi,u_2(\cdot)) \\
	&=\mathbb{E}\bigg\{\int_0^T \Big[\langle Qx,x\rangle +\Big\langle R_1\widetilde{R}_1^{-1}\big[B_1^\top(t)\bar{y}(t)+D_1^\top(t)\bar{z}(t)-\big(B_1^\top P \\
	&\qquad\quad+D_1^\top PC\big)\bar{x}(t)-D_1^\top PD_2u_2(t)-D_1^\top P\sigma\big],\widetilde{R}_1^{-1}\big[B_1^\top(t)\bar{y}(t)\\
    &\qquad\quad+D_1^\top(t)\bar{z}(t)-\big(B_1^\top P+D_1^\top PC\big)\bar{x}(t)-D_1^\top PD_2u_2(t)-D_1^\top P\sigma\big]\Big\rangle \\
	&\qquad\quad +\langle R_2u_2,u_2\rangle+\frac{2}{\gamma}\langle \hat{R}_0^{-1}\tilde{x},\tilde{x}\rangle \Big] \mathrm{d}t+\langle Gx(T),x(T)\rangle \bigg\} .
\end{aligned}
\end{equation}
It is clear that Problem \textbf{(LQ-L2)} is equivalent to the following problem:
		
Problem \textbf{(LQ-L2a)}: To find a control $\bar{u}_2(\cdot)\in \mathcal{U}_2[0,T]$ such that
$$
\tilde{J}_a^{wo}(\bar{u}_2(\cdot);\xi)=\inf_{u_2(\cdot)\in\, \mathcal{U}_2[0,T]}\tilde{J}_a^{wo}(u_2(\cdot);\xi),
$$
subject to $\xi \in\mathbb{R}^n$ and state equation (\ref{state L2}), where
\begin{equation}\label{J_L1-wo-a}
\begin{aligned}
	\tilde{J}_a^{wo}(u_2(\cdot);\xi)&:=\mathbb{E}\bigg\{\int_0^T \Big[-\langle \bar{\mathbb{Q}}\mathbb{X},\mathbb{X}\rangle -\langle \bar{\mathbb{B}}\mathbb{Y},\mathbb{Y}\rangle
    -\langle \bar{\mathbb{D}}\mathbb{Z},\mathbb{Z}\rangle-2\langle \mathbb{S}_1\mathbb{X},\mathbb{Y}\rangle\\
    &\qquad\quad-2\langle \mathbb{L}_1\mathbb{X},\mathbb{Z}\rangle-2\langle \mathbb{M}_1\mathbb{Y},\mathbb{Z}\rangle-\langle \mathbb{R}u_2,u_2\rangle-2\langle \mathbb{S}_2\mathbb{X},u_2\rangle\\
    &\qquad\quad-2\langle \mathbb{M}_2\mathbb{Y},u_2\rangle-2\langle \mathbb{L}_2\mathbb{Z},u_2\rangle-2\langle \mathbb{S}_3\mathbb{X},\sigma\rangle-2\langle \mathbb{M}_3\mathbb{Y},\sigma\rangle\\
	&\qquad\quad-2\langle \mathbb{L}_3\mathbb{Z},\sigma\rangle-2\langle D_2^\top PD_1RD_1^\top P\sigma,u_2\rangle\Big] \mathrm{d}t
    -\langle \bar{\mathbb{G}}\mathbb{X}(T),\mathbb{X}(T)\rangle \bigg\} ,
\end{aligned}
\end{equation}
\begin{equation*}
\left\{\begin{aligned}
	&R:=\widetilde{R}_1^{-1}R_1\widetilde{R}_1^{-1} ,\quad\mathbb{R}:=R_2+D_2^\top PD_1RD_1^\top PD_2  ,\\
	&\bar{\mathbb{Q}}:=
	\begin{pmatrix}
		Q & 0 & 0 & 0 & 0\\
		0 & \big(PB_1+C^\top PD_1\big)R(B_1^\top P+D_1^\top PC) & 0 & 0 & 0\\
		0 & 0 & 0 & 0 & 0\\
		0 & 0 & 0 & 0 & 0\\
		0 & 0 & 0 & 0 &0
	\end{pmatrix}, \\
	&\bar{\mathbb{B}}:=
	\begin{pmatrix}
		\frac{2}{\gamma}\hat{R}_0^{-1} & 0 & 0 & 0 & 0\\
		0 & 0 & 0 & 0 & 0\\
		0 & 0 & 0 & 0 & 0\\
		0 & 0 & 0 & B_1RB_1^\top & 0\\
		0 & 0 & 0 & 0 & 0
	\end{pmatrix}, \quad
    \bar{\mathbb{D}}:=
	\begin{pmatrix}
		0 & 0 & 0 & 0 & 0\\
		0 & 0 & 0 & 0 & 0\\
		0 & 0 & 0 & 0 & 0\\
		0 & 0 & 0 & D_1RD_1^\top & 0\\
		0 & 0 & 0 & 0 &0
	\end{pmatrix},\\
	&\mathbb{S}_1:=
	\begin{pmatrix}
		0 & 0 & 0 & 0 & 0\\
		0 & 0 & 0 & 0 & 0\\
		0 & 0 & 0 & 0 & 0\\
		0 & -B_1R\big(B_1^\top P+D_1^\top PC\big) & 0 & 0 & 0\\
		0 & 0 & 0 & 0 & 0
	\end{pmatrix}, \\
	&\mathbb{M}_1:=
	\begin{pmatrix}
		0 & 0 & 0 & 0 & 0\\
		0 & 0 & 0 & 0 & 0\\
		0 & 0 & 0 & 0 & 0\\
		0 & 0 & 0 & D_1RB_1^\top & 0\\
		0 & 0 & 0 & 0 & 0
	\end{pmatrix},\quad
	\bar{\mathbb{G}}:=
	\begin{pmatrix}
		G & 0 & 0 & 0 & 0\\
		0 & 0 & 0 & 0 & 0\\
		0 & 0 & 0 & 0 & 0\\
		0 & 0 & 0 & 0 & 0\\
		0 & 0 & 0 & 0 & 0
	\end{pmatrix}, \\
&\mathbb{L}_1:=
	\begin{pmatrix}
		0 & 0 & 0 & 0 & 0\\
		0 & 0 & 0 & 0 & 0\\
		0 & 0 & 0 & 0 & 0\\
		0 & -D_1R\big(B_1^\top P+D_1^\top PC\big) & 0 & 0 & 0\\
		0 & 0 & 0 & 0 & 0
	\end{pmatrix},\\\end{aligned}\right.
\end{equation*}
and
\begin{equation*}
\left\{\begin{aligned}
	&\mathbb{S}_2:=\big(0,D_2^\top PD_1R\big(B_1^\top P+D_1^\top PC\big),0,0,0\big),\\
    &\mathbb{M}_2:=\big(0,0,0,-D_2^\top PD_1RB_1^\top,0\big),\\
	&\mathbb{L}_2:=\big(0,0,0,-D_2^\top PD_1RD_1^\top,0\big),\\
    &\mathbb{S}_3:=\big(0,PD_1R\big(B_1^\top P+D_1^\top PC\big),0,0,0\big),\\
	&\mathbb{M}_3:=\big(0,0,0,-PD_1RB_1^\top,0\big),\\
    &\mathbb{L}_2:=\big(0,0,0,-PD_1RD_1^\top,0\big).
\end{aligned}\right.
\end{equation*}

To tackle this, we firstly need to address necessary convexity criteria of Problem \textbf{(LQ-L2a)}. By analogous reasoning as in Proposition \ref{prop-4.1}, we have the following result.
		
\begin{mypro}
Let (A1)-(A7) hold. For any $\xi\in \mathbb{R}^n$, $\tilde{J}_a^{wo}(u_2(\cdot);\xi)$ is convex (resp., strictly convex) in $u_2(\cdot)\in \mathcal{U}_2[0,T]$ if and only if $\tilde{J}_2^{''}(v_2(\cdot))$ is positive semidefinite (resp., positive definite), where
\begin{equation}
\begin{aligned}
	\tilde{J}_2^{''}(v_2(\cdot))&:=\mathbb{E}\bigg\{\int_0^T \Big[-\langle \bar{\mathbb{Q}}z_2^{''},z_2^{''}\rangle -\langle \bar{\mathbb{B}}y_2,y_2\rangle -\langle \bar{\mathbb{D}}z_2,z_2\rangle
    -2\langle \mathbb{S}_1z_2^{''},y_2\rangle \\
	&\qquad\qquad\quad -2\langle \mathbb{L}_1z_2^{''},z_2\rangle-2\langle \mathbb{M}_1y_2,z_2\rangle-\langle \mathbb{R}v_2,v_2\rangle-2\langle \mathbb{S}_2z_2^{''},v_2\rangle\\
	&\qquad\qquad\quad -2\langle \mathbb{M}_2y_2,v_2\rangle-2\langle \mathbb{L}_2z_2,v_2\rangle\Big] dt
    -\langle Gz_2^{''}(T),z_2^{''}(T)\rangle \bigg\},
\end{aligned}
\end{equation}
subject to
\begin{equation}\label{auxiliary control P3}
	\left\{\begin{aligned}
		\mathrm{d}z_2^{''}(t)&=\big[\mathbb{A}z_2^{''}+\mathbb{B}_1y_2+\mathbb{B}_3z_2+\mathbb{B}_2v_2\big]\mathrm{d}t
		+\big[\mathbb{C}z_2^{''}+\mathbb{D}_1y_2+\mathbb{D}_3z_2+\mathbb{D}_2v_2\big]\mathrm{d}W, \\
		\mathrm{d}y_2(t)&=\big[-\mathbb{A}^\top y_2-\mathbb{C}^\top z_2+\mathbb{Q}z_2^{''}+\mathbb{F}_2v_2\big]\mathrm{d}t+z_2(t)\mathrm{d}W, \\
		z_2^{''}(0)&=0, \quad y_2(T)=\mathbb{G}z_2^{''}(T).
	\end{aligned}\right.
\end{equation}
\end{mypro}

For further proof, we need the following assumption.

{\bf(A8)} (\romannumeral1) $\mathbb{R}(\cdot)\ll0$. (\romannumeral2) The map $v_2\mapsto \tilde{J}_2^{''}(v_2)$ is positive definite.
		
\begin{mythm}\label{thm-4.2}
Let (A1)-(A8) hold. Then Problem \textbf{(LQ-L2)} is solvable at $\xi\in\mathbb{R}^n$ if and only if the following Hamiltonian system:
\begin{equation}\label{Hamiltonian system}
\left\{\begin{aligned}					
    \mathrm{d}\hat{\mathbb{X}}(t)&=\big[\hat{\mathbb{A}}_1\hat{\mathbb{X}}+\hat{\mathbb{B}}_1\hat{\mathbb{Y}}+\hat{\mathbb{B}}_2\hat{\mathbb{Z}}+\hat{\mathbb{F}}\big]\mathrm{d}t
    +\big[\hat{\mathbb{C}}_1\hat{\mathbb{X}}+\hat{\mathbb{D}}_1\hat{\mathbb{Y}}+\hat{\mathbb{D}}_2\hat{\mathbb{Z}}+\hat{\Sigma}\big]\mathrm{d}W, \\
	\mathrm{d}\hat{\mathbb{Y}}(t)&=\big[-\hat{\mathbb{A}}_2^\top\hat{\mathbb{Y}}-\hat{\mathbb{C}}_2^\top\hat{\mathbb{Z}}+\hat{\mathbb{Q}}\hat{\mathbb{X}}+\hat{\Upsilon}\big]\mathrm{d}t+\hat{\mathbb{Z}}(t)\mathrm{d}W, \\
	\hat{\mathbb{X}}(0)&=\hat{\Xi}, \quad \hat{\mathbb{Y}}(T)=\hat{\mathbb{G}}\hat{\mathbb{X}}(T),
\end{aligned}\right.
\end{equation}
admits a unique solution $(\hat{\mathbb{X}}(\cdot),\hat{\mathbb{Y}}(\cdot),\hat{\mathbb{Z}}(\cdot))$, where
\begin{equation*}
\left\{\begin{aligned}
    &\hat{\mathbb{X}}:=(x^\top,\bar{x}^\top,\bar{q}^\top,\tilde{\bar{y}}^\top,\tilde{\bar{p}}^\top,\hat{\bar{y}}^\top,\hat{\bar{p}}^\top,\hat{\tilde{x}}^\top,\hat{\tilde{\bar{x}}}^\top,\hat{\tilde{\bar{q}}}^\top)^\top,\\ &\hat{\mathbb{Y}}:=(\hat{x}^\top,\hat{\bar{x}}^\top,\hat{\bar{q}}^\top,\hat{\tilde{\bar{y}}}^\top,\hat{\tilde{\bar{p}}}^\top,\bar{y}^\top,\bar{p}^\top,\tilde{x}^\top,\tilde{\bar{x}}^\top,\tilde{\bar{q}}^\top)^\top,\\
	&\hat{\mathbb{Z}}:=(\kappa^\top,\chi^\top,\mu^\top,\nu^\top,\eta^\top,\bar{z}^\top,\bar{\beta}^\top,\zeta^\top,\upsilon^\top,\varsigma^\top)^\top,\\
\end{aligned}\right.
\end{equation*}
\begin{equation*}
\left\{\begin{aligned}
	&\hat{\mathbb{A}}_1:=
	\begin{pmatrix}
		\mathbb{A}-\mathbb{B}_2\mathbb{R}^{-1}\mathbb{S}_2 & \mathbb{B}_2\mathbb{R}^{-1}\mathbb{F}_2^\top\\
		\mathbb{S}_1-\mathbb{M}_2^\top\mathbb{R}^{-1}\mathbb{S}_2 & \mathbb{A}+\mathbb{M}_2^\top\mathbb{R}^{-1}\mathbb{F}_2^\top
	\end{pmatrix}, \\
	&\hat{\mathbb{C}}_1:=
	\begin{pmatrix}
	 	\mathbb{C}-\mathbb{D}_2\mathbb{R}^{-1}\mathbb{S}_2 & \mathbb{D}_2\mathbb{R}^{-1}\mathbb{F}_2^\top\\
	 	\mathbb{L}_1-\mathbb{L}_2^\top\mathbb{R}^{-1}\mathbb{S}_2 & \mathbb{C}+\mathbb{L}_2^\top\mathbb{R}^{-1}\mathbb{F}_2^\top
	\end{pmatrix}, \\
	&\hat{\mathbb{A}}_2:=
	\begin{pmatrix}
	 	\mathbb{A}-\mathbb{B}_2\mathbb{R}^{-1}\mathbb{S}_2 & -\mathbb{B}_2\mathbb{R}^{-1}\mathbb{F}_2^\top\\
	 	-\mathbb{S}_1+\mathbb{M}_2^\top\mathbb{R}^{-1}\mathbb{S}_2 & \mathbb{A}+\mathbb{M}_2^\top\mathbb{R}^{-1}\mathbb{F}_2^\top
	\end{pmatrix}, \\
    &\hat{\mathbb{C}}_2:=
	\begin{pmatrix}
	 	\mathbb{C}-\mathbb{D}_2\mathbb{R}^{-1}\mathbb{S}_2 & -\mathbb{D}_2\mathbb{R}^{-1}\mathbb{F}_2^\top\\
	 	-\mathbb{L}_1+\mathbb{L}_2^\top\mathbb{R}^{-1}\mathbb{S}_2 & \mathbb{C}+\mathbb{L}_2^\top\mathbb{R}^{-1}\mathbb{F}_2^\top
	\end{pmatrix}, \\
    &\hat{\mathbb{B}}_1:=
	\begin{pmatrix}
			\mathbb{B}_2\mathbb{R}^{-1}\mathbb{B}_2^\top & \mathbb{B}_1-\mathbb{B}_2\mathbb{R}^{-1}\mathbb{M}_2\\
		-\mathbb{B}_1^\top+\mathbb{M}_2^\top\mathbb{R}^{-1}\mathbb{B}_2^\top & \bar{\mathbb{B}}-\mathbb{M}_2^\top\mathbb{R}^{-1}\mathbb{M}_2
	\end{pmatrix}, \\	
    &\hat{\mathbb{B}}_2:=
	\begin{pmatrix}
		\mathbb{B}_2\mathbb{R}^{-1}\mathbb{D}_2^\top & \mathbb{B}_3-\mathbb{B}_2\mathbb{R}^{-1}\mathbb{L}_2\\
		-\mathbb{D}_1^\top+\mathbb{M}_2^\top\mathbb{R}^{-1}\mathbb{D}_2^\top & \mathbb{M}_1^\top-\mathbb{M}_2^\top\mathbb{R}^{-1}\mathbb{L}_2
	\end{pmatrix}, \\
    &\hat{\mathbb{D}}_1:=
	\begin{pmatrix}
			\mathbb{D}_2\mathbb{R}^{-1}\mathbb{B}_2^\top & \mathbb{D}_1-\mathbb{D}_2\mathbb{R}^{-1}\mathbb{M}_2\\
			-\mathbb{B}_3^\top+\mathbb{L}_2^\top\mathbb{R}^{-1}\mathbb{B}_2^\top & \mathbb{M}_1-\mathbb{L}_2^\top\mathbb{R}^{-1}\mathbb{M}_2
	\end{pmatrix}, \\
	&\hat{\mathbb{D}}_2:=
	\begin{pmatrix}
		\mathbb{D}_2\mathbb{R}^{-1}\mathbb{D}_2^\top & \mathbb{D}_3-\mathbb{D}_2\mathbb{R}^{-1}\mathbb{L}_2\\
		-\mathbb{D}_3^\top+\mathbb{L}_2^\top\mathbb{R}^{-1}\mathbb{D}_2^\top & \bar{\mathbb{D}}^\top-\mathbb{L}_2^\top\mathbb{R}^{-1}\mathbb{L}_2
	\end{pmatrix}, \\
	&\hat{\mathbb{Q}}:=
	\begin{pmatrix}
		\bar{\mathbb{Q}}-\mathbb{S}_2^\top\mathbb{R}^{-1}\mathbb{S}_2 & -\mathbb{Q}^\top+\mathbb{S}_2^\top\mathbb{R}^{-1}\mathbb{F}_2^\top\\
		\mathbb{Q}-\mathbb{F}_2\mathbb{R}^{-1}\mathbb{S}_2 & \mathbb{F}\mathbb{R}^{-1}\mathbb{F}_2^\top
	\end{pmatrix}, \\
	&\hat{\mathbb{F}}:=
	\begin{pmatrix}
		\mathbb{F}_1-\mathbb{B}_2\mathbb{R}^{-1}D_2^\top PD_1RD_1^\top P\sigma \\
		(\mathbb{M}_3^\top-\mathbb{M}_2^\top \mathbb{R}^{-1}D_2^\top PD_1RD_1^\top P)\sigma
	\end{pmatrix},\\
	&\hat{\Sigma}:=
	\begin{pmatrix}
		\Sigma-\mathbb{D}_2\mathbb{R}^{-1}D_2^\top PD_1RD_1^\top P\sigma \\
		(\mathbb{L}_3^\top-\mathbb{L}_2^\top \mathbb{R}^{-1}D_2^\top PD_1RD_1^\top P)\sigma
	\end{pmatrix},\quad \hat{\Xi}:=(\Xi^\top,0)^\top,\\
	&\hat{\Upsilon}:=
	\begin{pmatrix}
		(\mathbb{S}_3^\top-\mathbb{S}_2^\top \mathbb{R}^{-1}D_2^\top PD_1RD_1^\top P)\sigma\\
		\Upsilon-\mathbb{F}_2\mathbb{R}^{-1}D_2^\top PD_1RD_1^\top P\sigma
	\end{pmatrix},\quad
    \hat{\mathbb{G}}:=
	\begin{pmatrix}
		-\bar{\mathbb{G}} & -\mathbb{G}^\top\\
		\mathbb{G} & 0
	\end{pmatrix},
\end{aligned}\right.
\end{equation*}
and the superscript of vector elements ``$\,\,\hat{\cdot}\,$" represents adjoint part of the corresponding processes. Moreover, the optimal control is given by
\begin{equation}\label{u_2*}
	\begin{aligned}
		\bar{u}_2(t)&=\mathbb{R}^{-1}\biggl\{\big(B_2^\top-D_2^\top PD_1\widetilde{R}_1^{-1} B_1^\top \big)\big(\hat{x}(t)+\hat{\bar{x}}(t)\big)\\
        &\qquad\quad +\big(D_2^\top-D_2^\top PD_1\widetilde{R}_1^{-1}D_1^\top \big)\big(\kappa(t)+\chi(t)\big) \\
		&\qquad\quad +\Big[B_2^\top P+D_2^\top PC-D_2^\top PD_1\widetilde{R}_1^{-1}\big(B_1^\top P+D_1^\top PC\big) \Big]\hat{\tilde{\bar{x}}}(t) \\
		&\qquad\quad -D_2^\top PD_1R\big(B_1^\top P+D_1^\top PC\big)\bar{x}(t)+D_2^\top PD_1RB_1^\top\tilde{\bar{x}}(t) \\
		&\qquad\quad +D_2^\top PD_1RD_1^\top\upsilon(t)-D_2^\top PD_1RD_1^\top P\sigma(t)\biggr\},\quad t\in[0,T],
	\end{aligned}
\end{equation}
where $\bar{x}(\cdot)$, $\hat{\tilde{\bar{x}}}(\cdot)$; $\hat{x}(\cdot)$,  $\hat{\bar{x}}(\cdot)$, $\tilde{\bar{x}}(\cdot)$; $\kappa(\cdot)$, $\chi(\cdot)$, $\upsilon(\cdot)$ are elements of $\hat{\mathbb{X}}(\cdot)$, $\hat{\mathbb{Y}}(\cdot)$, $\hat{\mathbb{Z}}(\cdot)$, respectively.
\end{mythm}

\begin{Remark}
Problem \textbf{(LQ-L2)} is a stochastic LQ optimal control problem of FBSDEs in an augmented space. Because the cost functional contains inter-state and cross-term between state and control, and the backward equation contains the control $u_2(\cdot)$, so that the forward equation and backward equation are not dual formally after dimension extension. We find that the corresponding elements on both sides of the main diagonal of the matrix $\hat{\mathbb{A}}_1$ and $\hat{\mathbb{A}}_2$ have opposite signs, and so do the matrix $\hat{\mathbb{C}}_1$ and $\hat{\mathbb{C}}_2$.
\end{Remark}
	
\begin{mypro}
If Riccati equation
\begin{equation}\label{R-4}
\left\{\begin{aligned}
	&\dot{\hat{P}}+\hat{P}\hat{\mathbb{A}}_1+\hat{\mathbb{A}}_2^\top\hat{P}+\hat{P}\hat{\mathbb{B}}_1\hat{P}-\hat{\mathbb{Q}}
    +(\hat{\mathbb{C}}_2^\top+\hat{P}\hat{\mathbb{B}}_2)(I-\hat{P}\hat{\mathbb{D}}_2)^{-1}(\hat{P}\hat{\mathbb{C}}_1+\hat{P}\hat{\mathbb{D}}_1\hat{P})=0, \\
	&\hat{P}(T)=\hat{\mathbb{G}},
\end{aligned}\right.
\end{equation}
admits a unique solution $\hat{P}(\cdot)\in C([0,T];\mathbb{R}^{10n\times10n})$ over $[0,T]$, then BSDE
\begin{equation}
\left\{\begin{aligned}\label{B-4}			
    \mathrm{d}\hat{\varphi}(t)&=-\Big\{\big[\hat{\mathbb{A}}_2^\top+\hat{P}\hat{\mathbb{B}}_1+(\hat{\mathbb{C}}_2^\top
    +\hat{P}\hat{\mathbb{B}}_2)(I-\hat{P}\hat{\mathbb{D}}_2)^{-1}\hat{P}\hat{\mathbb{D}}_1\big]\hat{\varphi}\\
    &\qquad+(\hat{\mathbb{C}}_2^\top+\hat{P}\hat{\mathbb{B}}_2)(I-\hat{P}\hat{\mathbb{D}}_2)^{-1}(\hat{P}\hat{\Sigma}+\hat{\Lambda})+\hat{P}\hat{\mathbb{F}}-\hat{\Upsilon}\Big\} \mathrm{d}t+\hat{\Lambda}(t)\mathrm{d}W, \\
	\hat{\varphi}(T)&=0,
\end{aligned}\right.
\end{equation}
admits a unique solution $(\hat{\varphi}(\cdot),\hat{\Lambda}(\cdot))\in L_{\mathbb{F}}^2(\Omega;C([0,T];\mathbb{R}^{10n}))\times L_{\mathbb{F}}^2(0,T;\mathbb{R}^{10n})$, then the wellposedness of system (\ref{Hamiltonian system}) is obtained.
\end{mypro}
	
\begin{proof}
Similar to the proof of Proposition \ref{prop-3.5}.
\end{proof}
	
\section{Robust Stackelberg equilibrium}
	
As above mentioned, once (\ref{R-4}) admits a solution $\hat{P}(\cdot)$, the BSDE (\ref{B-4}) admits a unique adapted solution $(\hat{\varphi}(\cdot),\hat{\Lambda}(\cdot))$. Then, the equation
\begin{equation}
\left\{\begin{aligned}\label{LSDE}
	\mathrm{d}\hat{\mathbb{X}}(t)&=(\tilde{\mathbb{A}}\hat{\mathbb{X}}+\tilde{\mathbb{B}})\mathrm{d}t+(\tilde{\mathbb{C}}\hat{\mathbb{X}}+\tilde{\mathbb{D}})\mathrm{d}W, \\
	\hat{\mathbb{X}}(0)&=\hat{\Xi},
\end{aligned}\right.
\end{equation}
admits a unique solution $\hat{\mathbb{X}}(\cdot)$, where
\begin{equation*}
\left\{\begin{aligned}
   &\tilde{\mathbb{A}}:=\hat{\mathbb{A}}_1+\hat{\mathbb{B}}_1\hat{P}+\hat{\mathbb{B}}_2(I-\hat{P}\hat{\mathbb{D}}_2)^{-1}\hat{P}(\hat{\mathbb{C}}_1+\hat{\mathbb{D}}_1\hat{P}), \\			&\tilde{\mathbb{B}}:=\hat{\mathbb{B}}_1\hat{\varphi}+\hat{\mathbb{B}}_2(I-\hat{P}\hat{\mathbb{D}}_2)^{-1}(\hat{P}\hat{\mathbb{D}}_1\hat{\varphi}+\hat{P}\hat{\Sigma}+\hat{\Lambda})+\hat{\mathbb{F}}, \\
   &\tilde{\mathbb{C}}:=\hat{\mathbb{C}}_1+\hat{\mathbb{D}}_1\hat{P}+\hat{\mathbb{D}}_2(I-\hat{P}\hat{\mathbb{D}}_2)^{-1}\hat{P}(\hat{\mathbb{C}}+\hat{\mathbb{D}}_1\hat{P}), \\		&\tilde{\mathbb{D}}:=\hat{\mathbb{D}}_1\hat{\varphi}+\hat{\mathbb{D}}_2(I-\hat{P}\hat{\mathbb{D}}_2)^{-1}(\hat{P}\hat{\mathbb{D}}_1\hat{\varphi}+\hat{P}\hat{\Sigma}+\hat{\Lambda})+\hat{\Sigma}.
\end{aligned}\right.
\end{equation*}
Furthermore, the second equation in (\ref{Hamiltonian system}) admits a unique solution $(\hat{\mathbb{Y}}(\cdot),\hat{\mathbb{Z}}(\cdot))$.
		
For convenience, we introduce some notations:
\begin{equation*}
\left\{\begin{aligned}
	M_1&:=(1,0,0,0,0,0,0,0,0,0),\quad M_2:=(0,1,0,0,0,0,0,0,0,0), \\
	M_3&:=(1,1,0,0,0,0,0,0,0,0),\quad M_4:=(0,0,0,0,0,1,0,0,0,0),\\
    M_5&:=(0,0,0,0,0,0,1,0,0,0),\quad M_6:=(0,0,0,0,0,0,0,1,0,0),\\
    M_7&:=(0,0,0,0,0,0,0,0,1,0), \\
    \hat{P}_M^1&:=B_1^\top M_4\hat{P}+D_1^\top M_4(I-\hat{P}\hat{\mathbb{D}}_2)^{-1}(\hat{P}\hat{\mathbb{C}}_1+\hat{P}\hat{\mathbb{D}}_1\hat{P})\\
    &\quad -\big(B_1^\top P+D_1^\top PC\big)M_2-D_1^\top PD_2\mathbb{R}^{-1}\hat{P}_M^2, \\
    \hat{P}_M^2&:=\big(B_2^\top-D_2^\top PD_1\widetilde{R}_1^{-1} B_1^\top \big)M_3\hat{P}\\
    &\quad +\Big[B_2^\top P+D_2^\top PC-D_2^\top PD_1\widetilde{R}_1^{-1}\big(B_1^\top P+D_1^\top PC\big) \Big]M_7 \\
    &\quad -D_2^\top PD_1R\big(B_1^\top P+D_1^\top PC\big)M_2+D_2^\top PD_1RB_1^\top M_7\hat{P} \\
    &\quad +\big(D_2^\top M_3-D_2^\top PD_1\widetilde{R}_1^{-1}D_1^\top M_3+D_2^\top PD_1RD_1^\top M_7\big)\\
    &\quad \times(I-\hat{P}\hat{\mathbb{D}}_2)^{-1}(\hat{P}\hat{\mathbb{C}}_1+\hat{P}\hat{\mathbb{D}}_1\hat{P}), \\
    \hat{\varphi}_M^1&:=B_1^\top M_4\hat{\varphi}+D_1^\top M_4(I-\hat{P}\hat{\mathbb{D}}_2)^{-1}(\hat{P}\hat{\mathbb{D}}_1\hat{\varphi}+\hat{P}\hat{\Lambda})\\
    &\quad -D_1^\top P\sigma-D_1^\top PD_2\mathbb{R}^{-1}\hat{\varphi}_M^2, \\
    \hat{\varphi}_M^2&:=\big(B_2^\top-D_2^\top PD_1\widetilde{R}_1^{-1} B_1^\top \big)M_3\hat{\varphi}+D_2^\top PD_1RB_1^\top M_7\hat{\varphi}\\
    &\quad -D_2^\top PD_1RD_1^\top P\sigma+\big(D_2^\top M_3-D_2^\top PD_1\widetilde{R}_1^{-1}D_1^\top M_3 \\
    &\quad +D_2^\top PD_1RD_1^\top M_7\big)(I-\hat{P}\hat{\mathbb{D}}_2)^{-1}(\hat{P}\hat{\mathbb{D}}_1\hat{\varphi}+\hat{P}\hat{\Lambda}). \\
\end{aligned}\right.
\end{equation*}
		
\begin{mypro}
Let (A1)-(A8) hold. Suppose Riccati equation (\ref{R-4}) admits a unique solution $\hat{P}(\cdot)$ and the Lyapunov equation
\begin{equation}\label{value function-Lyapunov equation}
\left\{\begin{aligned}
	&\dot{\mathbb{L}}+\mathbb{L}\tilde{\mathbb{A}}+\tilde{\mathbb{A}}^\top\mathbb{L}+\tilde{\mathbb{C}}^\top\mathbb{L}\tilde{\mathbb{C}}+M_1^\top QM_1+{{}\hat{P}_M^1}^\top R\hat{P}_M^1+{{}\hat{P}_M^2}^\top\mathbb{R}^{-1}R_2\mathbb{R}^{-1}\hat{P}_M^2=0, \\
	&\mathbb{L}(T)=M_1^\top GM_1,
\end{aligned}\right.
\end{equation}
admits a unique solution $\mathbb{L}(\cdot)$ and the following BSDE
\begin{equation}\label{value function-BSDE}
\left\{\begin{aligned}
	\mathrm{d}\psi(t)&=-\Big[\tilde{\mathbb{A}}^\top\psi+\tilde{\mathbb{C}}^\top\bar{\psi}+\mathbb{L}\tilde{\mathbb{B}}+\tilde{\mathbb{C}}^\top\mathbb{L}\tilde{\mathbb{D}}
    +{{}\hat{P}_M^1}^\top R\hat{\varphi}_M^1\\
    &\qquad +{{}\hat{P}_M^2}^\top\mathbb{R}^{-1}R_2\mathbb{R}^{-1}\hat{\varphi}_M^2\Big]\mathrm{d}t+\bar{\psi}(t)\mathrm{d}W, \\
	\psi(T)&=0,
\end{aligned}\right.
\end{equation}
admits a unique solution $(\psi(\cdot),\bar{\psi}(\cdot))$. Let matrix-valued process $N(\cdot)$ be the solution to
\begin{equation*}
\left\{\begin{aligned}
	\mathrm{d}N(t)&=\tilde{\mathbb{A}}(t)N(t)\mathrm{d}t+\tilde{\mathbb{C}}(t)N(t)\mathrm{d}W, \\
	N(0)&=I.
\end{aligned}\right.
\end{equation*}
Then, the forward state $\hat{\mathbb{X}}(\cdot)$ satisfying (\ref{LSDE}) can be written as
\begin{equation}\label{final state}
\begin{aligned}
	\hat{\mathbb{X}}(t)&=N(t)\hat{\Xi}+N(t)\int_0^t N^{-1}(s)\big[\tilde{\mathbb{B}}(s)-\tilde{\mathbb{C}}(s)\tilde{\mathbb{D}}(s)\big]\mathrm{d}s\\
    &\quad +N(t)\int_0^t N^{-1}(s)\tilde{\mathbb{D}}(s)\mathrm{d}W(s).
\end{aligned}
\end{equation}
The robust Stackelberg equilibrium can be designed as
\begin{equation}\label{robust strategies}
	\left\{
	\begin{aligned}
		\bar{u}_1(t)&=\widetilde{R}_1^{-1}(\hat{P}_M^1\hat{\mathbb{X}}+\hat{\varphi}_M^1), \\
		\bar{u}_2(t)&=\mathbb{R}^{-1}(\hat{P}_M^2\hat{\mathbb{X}}+\hat{\varphi}_M^2), \\
		\bar{f}(t)&=-\frac{2}{\alpha}R_0^{-1}M_5(\hat{P}\hat{\mathbb{X}}+\hat{\varphi}), \\
		\bar{f}_2(t)&=\frac{2}{\gamma}\hat{R}_0^{-1}M_6(\hat{P}\hat{\mathbb{X}}+\hat{\varphi}),\quad t\in[0,T],
	\end{aligned}
	\right.
\end{equation}
where $f(t)\equiv f_1(t)+f_2(t)$. Moreover, the optimal cost functional can be represented by
\begin{equation}\label{value function}
	\begin{aligned}
		V(\xi)&:=J(\xi;\bar{u}_1(\cdot),\bar{u}_2(\cdot))=\mathbb{E}\int_0^T\Big[{{}\hat{\varphi}_M^1}^\top R\hat{\varphi}_M^1+{{}\hat{\varphi}_M^2}^\top\mathbb{R}^{-1}R_2\mathbb{R}^{-1}\hat{\varphi}_M^2 \\
		&+\tilde{\mathbb{D}}^\top\mathbb{L}\tilde{\mathbb{D}}+2(\tilde{\mathbb{B}}^\top\psi+\tilde{\mathbb{D}}^\top\bar{\psi})\Big]\mathrm{d}t
		+\hat{\Xi}^\top\mathbb{L}(0)\hat{\Xi}+2\hat{\Xi}^\top\psi(0).
	\end{aligned}
	\end{equation}
\end{mypro}

\begin{proof}
Obviously, $N^{-1}(t)$ exists for all $t\geqslant0$ (Chapter 6 of Yong and Zhou \cite{Yong99}). It is easy to check $\hat{\mathbb{X}}(\cdot)$ given by (\ref{final state}) satisfies (\ref{LSDE}).
Based on Propositions \ref{prop-3.1}, \ref{prop-3.4} and Theorems \ref{thm-4.1}, \ref{thm-4.2}, we can determine the mappings
\begin{equation}
\left\{\begin{aligned}
	&\bar{f}(t)=\bar{\alpha}_1\left[\xi,\bar{u}_1(\cdot),\bar{u}_2(\cdot)\right](t)=-\frac{2}{\alpha}R_0^{-1}\bar{p}(t)=-\frac{2}{\alpha}R_0^{-1}M_5\big(\hat{P}\hat{\mathbb{X}}+\hat{\varphi}\big), \\
	&\bar{u}_1(t)=\bar{\alpha}_2\left[ \xi,\bar{u}_2(\cdot)\right] (t)=\widetilde{R}_1^{-1}\big(\hat{P}_M^1\hat{\mathbb{X}}+\hat{\varphi}_M^1\big), \\
	&\bar{f}_2(t)=\bar{\alpha}_3\left[ \xi,\bar{u}_2(\cdot)\right] (t)=\frac{2}{\gamma}\hat{R}_0^{-1}\tilde{x}(t)=\frac{2}{\gamma}\hat{R}_0^{-1}M_6\big(\hat{P}\hat{\mathbb{X}}+\hat{\varphi}\big), \\
	&\bar{u}_2(t)=\mathbb{R}^{-1}(\hat{P}_M^2\hat{\mathbb{X}}+\hat{\varphi}_M^2).
\end{aligned}\right.
\end{equation}
From above, we obtain the robust strategies (\ref{robust strategies}). Furthermore, applying It\^o's formula to $\langle \mathbb{L}\hat{\mathbb{X}}(\cdot),\hat{\mathbb{X}}(\cdot)\rangle$ and $\langle \psi(\cdot),\hat{\mathbb{X}}(\cdot)\rangle$, respectively, and noting $\mathbb{L}=\mathbb{L}^\top$, based on (\ref{value function-Lyapunov equation}) and (\ref{value function-BSDE}), we have (\ref{value function}).
The proof is complete.
\end{proof}

Because (A4) can ensure Riccati equation (\ref{R-4}) admits a solution $P_1(\cdot)$, we can skip it. Now, we may focus on the solvability of Riccati equations (\ref{R-2}) (\ref{R-3}) and (\ref{R-4}). For convenience, we discuss the solvability of (\ref{R-2}), (\ref{R-3}) and (\ref{R-4}) in a special but nontrivial case: $C(\cdot)\equiv0$, $D_1(\cdot)\equiv0$, $D_2(\cdot)\equiv0$.
		
In this case, $\hat{C}(\cdot)$, $\hat{B}_3(\cdot)$, $\hat{D}_1(\cdot)$, $\mathbb{C}(\cdot)$, $\mathbb{B}_3(\cdot)$, $\mathbb{D}_1(\cdot)$, $\hat{\mathbb{C}}_1(\cdot)$, $\hat{\mathbb{C}}_2(\cdot)$, $\hat{\mathbb{B}}_2(\cdot)$, $\hat{\mathbb{D}}_1(\cdot)$ all disappear. Then Riccati equations (\ref{R-2}), (\ref{R-3}) and (\ref{R-4}) become, respectively,
\begin{equation}\label{R-2'}
	\dot{P}_2+P_2\hat{A}_1+\hat{A}_2^\top P_2+P_2\hat{B}_1P_2-\hat{Q}=0 ,\quad P_2(T)=\hat{G},
\end{equation}
\begin{equation}\label{R-3'}
	\dot{P}_3+P_3\mathbb{A}+\mathbb{A}^\top P_3+P_3\mathbb{B}_1P_3-\mathbb{Q}=0 ,\quad P_3(T)=\mathbb{G},
\end{equation}
\begin{equation}\label{R-4'}
	\dot{\hat{P}}+\hat{P}\hat{\mathbb{A}}_1+\hat{\mathbb{A}}_2^\top\hat{P}+\hat{P}\hat{\mathbb{B}}_1\hat{P}-\hat{\mathbb{Q}}=0 ,\quad \hat{P}(T)=\hat{\mathbb{G}}.
\end{equation}
			
\begin{mypro}
Let (A1)-(A8) hold, $C(\cdot)\equiv0$, $D_1(\cdot)\equiv0$, $D_2(\cdot)\equiv0$. For any $s\in[0,T]$, let $\Psi_i(\cdot,s)$ be the solutions, respectively, to the following \emph{ordinary differential equations} (ODEs):
\begin{equation}
	\left\{\begin{aligned}
		&\frac{\mathrm{d}}{\mathrm{d} t} \Psi_i(t,s)=\check{\mathbb{A}}_i(t)\Psi_i(t,s),\quad t\in[s,T], \\
		&\Psi_i(s,s)=I,\quad i=1,2,3. \end{aligned}\right.
\end{equation}
Suppose that
$$
\left[ \begin{pmatrix}
	0&I
\end{pmatrix}\Psi_i(T,t)\begin{pmatrix}
	0\\
	I
\end{pmatrix}\right] ^{-1},\quad i=1,2,3
$$
are $L^1$ bounded. Then Riccati equations (\ref{R-2'})-(\ref{R-4'}) admit unique solutions $P_2(\cdot)$, $P_3(\cdot)$, $\hat{P}(\cdot)$, respectively, which are explicitly given by
\begin{equation}\label{Riccatis' solutions}
	\left\{\begin{aligned}
		&P_2(t)=\hat{G}-\left[ \begin{pmatrix}
			0&I
		\end{pmatrix}\Psi_1(T,t)\begin{pmatrix}
			0\\
			I
		\end{pmatrix}\right] ^{-1}\begin{pmatrix}
			0&I
		\end{pmatrix}\Psi_1(T,t)\begin{pmatrix}
			I\\
			0
		\end{pmatrix}, \\
		&P_3(t)=\mathbb{G}-\left[  \begin{pmatrix}
			0&I
		\end{pmatrix}\Psi_2(T,t)\begin{pmatrix}
			0\\
			I
		\end{pmatrix}\right] ^{-1}\begin{pmatrix}
			0&I
		\end{pmatrix}\Psi_2(T,t)\begin{pmatrix}
			I\\
			0
		\end{pmatrix}, \\
		&\hat{P}(t)=\hat{\mathbb{G}}-\left[  \begin{pmatrix}
			0&I
		\end{pmatrix}\Psi_3(T,t)\begin{pmatrix}
			0\\
			I
		\end{pmatrix}\right] ^{-1}\begin{pmatrix}
			0&I
		\end{pmatrix}\Psi_3(T,t)\begin{pmatrix}
			I\\
			0
		\end{pmatrix},
	\end{aligned}\right.
	\quad t\in[0,T],
\end{equation}
where
\begin{equation*}
	\left\{\begin{aligned}
		&\check{\mathbb{A}}_1(\cdot):=
		\begin{pmatrix}
			\hat{A}_1(\cdot)+\hat{B}_1(\cdot)\hat{G} & \hat{B}_1(\cdot)\\
			-\hat{G}\hat{A}_1(\cdot)-\hat{A}_2^\top(\cdot)\hat{G}-\hat{G}\hat{B}_1(\cdot)\hat{G}+\hat{Q}(\cdot)&-\hat{A}_2^\top(\cdot)-\hat{G}\hat{B}_1(\cdot)
		\end{pmatrix}, \\
		&\check{\mathbb{A}}_2(\cdot):=
		\begin{pmatrix}
			\mathbb{A}(\cdot)+\mathbb{B}_1(\cdot)\mathbb{G} & \mathbb{B}_1(\cdot)\\
			-\mathbb{G}\mathbb{A}(\cdot)-\mathbb{A}^\top(\cdot)\mathbb{G}-\mathbb{G}\mathbb{B}_1(\cdot)\mathbb{G}+\mathbb{Q}(\cdot)&-\mathbb{A}^\top(\cdot)-\mathbb{G}\mathbb{B}_1(\cdot)
		\end{pmatrix}, \\
		&\check{\mathbb{A}}_3(\cdot):=
		\begin{pmatrix}
			\hat{\mathbb{A}}_1(\cdot)+\hat{\mathbb{B}}_1(\cdot)\hat{\mathbb{G}} & \hat{\mathbb{B}}_1(\cdot)\\
			-\hat{\mathbb{G}}\hat{\mathbb{A}}_1(\cdot)-\hat{\mathbb{A}}_2^\top(\cdot)\hat{\mathbb{G}}-\hat{\mathbb{G}}\hat{\mathbb{B}}_1(\cdot)\hat{\mathbb{G}}+\hat{\mathbb{Q}}(\cdot)
			&-\hat{\mathbb{A}}_2^\top(\cdot)-\hat{\mathbb{G}}\hat{\mathbb{B}}_1(\cdot)
		\end{pmatrix}.
	\end{aligned}\right.
\end{equation*}

Then, the solution to the Lyapunov equation (\ref{value function-Lyapunov equation}) is characterized as
\begin{equation}\label{L}
	\begin{aligned}
		\mathbb{L}(t)&=e^{\int_t^T\big[\hat{\mathbb{A}}_1(r)+\hat{\mathbb{B}}_1(r)\hat{P}(r)\big]^\top \mathrm{d}r}M_1^\top GM_1e^{\int_t^T\big[\hat{\mathbb{A}}_1(r)+\hat{\mathbb{B}}_1(r)\hat{P}(r)\big]\mathrm{d}r}\\
		&\quad +\int_t^Te^{\int_t^s\big[\hat{\mathbb{A}}_1(r)+\hat{\mathbb{B}}_1(r)\hat{P}(r)\big]^\top \mathrm{d}r}\Big({{}\hat{P}_M^1}^\top R\hat{\varphi}_M^1\\
        &\quad +{{}\hat{P}_M^2}^\top\mathbb{R}^{-1}R_2\mathbb{R}^{-1}\hat{\varphi}_M^2\Big)e^{\int_t^s\big[\hat{\mathbb{A}}_1(r)+\hat{\mathbb{B}}_1(r)\hat{P}(r)\big]\mathrm{d}r}\mathrm{d}s,
	\end{aligned}
\end{equation}
and the solution to BSDE (\ref{value function-BSDE}) is given by
\begin{equation}\label{psi}
	\psi(t)=\mathbb{E}\left\{ \int_t^T\Big[\mathbb{L}(\hat{\mathbb{B}}_1\hat{\varphi}
    +\hat{\mathbb{F}})+{{}\hat{P}_M^1}^\top R\hat{\varphi}_M^1+{{}\hat{P}_M^2}^\top\mathbb{R}^{-1}R_2\mathbb{R}^{-1}\hat{\varphi}_M^2\Big]\Psi_4(s)\mathrm{d}s\Big\vert\mathcal{F}_t\right\},
\end{equation}
and $\Psi_4(\cdot)$ satisfies
\begin{equation*}
	\left\{\begin{aligned}
		\mathrm{d}\Psi_4(t)&=\big(\hat{\mathbb{A}}_1+\hat{\mathbb{B}}_1\hat{P}\big)\Psi_4(t)\mathrm{d}t, \\
		\Psi_4(T)&=I.
	\end{aligned}\right.
\end{equation*}
\end{mypro}
			
\begin{proof}
Because these three Riccati equations are with similar type, while the coefficients might be different, we take (\ref{R-2'}) as an example. Define $\Pi(t):=P_2(t)-\hat{G}$ for $t\in[0,T]$, thus
\begin{equation}
\left\{\begin{aligned}
		&\dot{\Pi}+\Pi(\hat{A}_1+\hat{B}_1\hat{G})+\big(\hat{A}_2^\top+\hat{G}^\top\hat{B}_1^\top\big)\Pi+\Pi\hat{B}_1\Pi+\hat{G}\hat{A}_1+\hat{A}_2^\top\hat{G}+\hat{G}\hat{B}_1\hat{G}-\hat{Q}=0, \\
		&\Pi(T)=0.
\end{aligned}\right.
\end{equation}
Combining Theorem 5.3 in Yong \cite{Yong06}, we obtain  (\ref{Riccatis' solutions}).

Under conditions $C(\cdot)\equiv0$, $D_1(\cdot)\equiv0$, $D_2(\cdot)\equiv0$, the Lyapunov equation (\ref{value function-Lyapunov equation}) is a backward differential Sylvester equation. See the paper \cite{Huang-Wang-Wu22} for detailed proof.
\end{proof}

Before the end of this section, we give an example.

\begin{Example}
	Suppose there are two producers in the market supplying the same product, and producer 1 has a relatively small production scale, which is called the follower. And producer 2 has a larger production scale, market share and production experience, which is the so-called leader. $u_1(\cdot)$, $u_2(\cdot)$ represents the output of producer 1 and producer 2, respectively. $x(\cdot)$ is the total quantity of the product in the market, which satisfies the following SDE:
\begin{equation}
\left\{\begin{aligned}
	dx(t)&=\big[(1-a)x(t)+u_1(t)+u_2(t)+f_1(t)+f_2(t)\big]dt\\
    &\quad +\big[cx(t)+u_1(t)+u_2(t)\big]dW(t), \\
	x(0)&=\xi,
\end{aligned}\right.
\end{equation}
where $f_1(\cdot)$ and $f_2(\cdot)$ are both unknown to the producer 1. For the producer 2, $f_1(\cdot)$ is fully observed and $f_2(\cdot)$ is an unknown disturbance. The parameter $a\in(0,1)$ is purchase rate of the product, $c\in\mathbb{R}$ represents some random environmental effect. The cost functional is the brand influence of Producer 2:
\begin{equation}
    J(\xi;u_1(\cdot),u_2(\cdot))= \mathbb{E}\bigg\{ \int_0^T\Big[ qx^2(t) + r_1u_1^2(t)+ r_2u^2_2(t)\Big]\mathrm{d}t+gx^2(T)\bigg\},
\end{equation}
where $q$, $g$, $r_1$, $r_2$ are assumed to be constant. The Producer 1 wants to minimize $J(\xi;u_1(\cdot),u_2(\cdot))$ by choosing $\bar{u}_1(\cdot)$ first, then Producer 2 wishes to maximize it buy selecting $\bar{u}_2(\cdot)$, which constitutes a zero-sum SLQ Stackelberg differential game. By the previous results we obtained, we can get the robust Stackelberg equilibrium:
\begin{equation*}
\begin{aligned}
    \bar{u}_1(t)&=\frac{1}{r_1+P}\big[\bar{y}(t)+\bar{z}(t)-P(1+c)\bar{x}(t)-P\bar{u}_2(t)\big]\vee 0,\\
    \bar{u}_2(t)&=\frac{1}{r_2+\frac{r_1}{(1+\frac{r_1}{P})^2}}\bigg[\frac{r_1}{r_1+P}(\hat{x}(t)+\hat{\bar{x}}(t)+\kappa(t)+\chi(t))+\frac{r_1P}{r_1+P}(1+c)\hat{\tilde{\bar{x}}}(t) \\
    &\qquad-\frac{r_1P^2}{(r_1+P)^2}(1+c)\bar{x}(t)+\frac{r_1P}{(r_1+P)^2}(\tilde{\bar{x}}(t)+\upsilon(t))\bigg]\vee 0,\quad t\in[0,T],
\end{aligned}
\end{equation*}
where $\bar{x}(\cdot)$, $\hat{\tilde{\bar{x}}}(\cdot)$; $\hat{x}(\cdot)$,  $\hat{\bar{x}}(\cdot)$, $\tilde{\bar{x}}(\cdot)$, $\bar{y}(\cdot)$ and $\bar{z}(\cdot)$,  $\kappa(\cdot)$, $\chi(\cdot)$, $\upsilon(\cdot)$ are elements of $\hat{\mathbb{X}}(\cdot)$, $\hat{\mathbb{Y}}(\cdot)$ and $\hat{\mathbb{Z}}(\cdot)$ which are adapted solutions of Hamiltonian system (\ref{Hamiltonian system}) after replacing the corresponding parameter. And $P(\cdot)$ is the solution to the following equation:
\begin{equation}\label{bode}
\left\{\begin{aligned}
    &\dot{P}(t)+\big[2(1-a)+c^2\big]P(t)-\frac{P^2(t)(1+c)^2}{r_1+P(t)}+q=0,\\
    &P(T)=g.
\end{aligned}\right.
 \end{equation}
In the above, since $u_i(\cdot), i=1,2$ is the output, which is positive, so we need to take a larger value between the results and zero.

Intuitively, regardless of other factors, the increase in the output of Producer 1 will weaken the brand effect of Producer 2 to a certain extent, so the output of Producer 1 is negatively correlated with that of Producer 2. The output of Producer 2 promotes its brand. Through the above analysis, we can get $r_1< 0$, $r_2> 0$, which shows that indefinite weighting matrices are necessary. We just need to make sure $r_1+P\gg 0$ and $r_2+\frac{r_1}{(1+\frac{r_1}{P})^2}\ll 0$. Because the robust Stackelberg equilibrium depends on the solution to the Hamiltonian system (\ref{Hamiltonian system}), it is insensitive to random perturbations in the state equation, which shows the robustness of the robust Stackelberg equilibrium.

Consider this problem with the parameters' values are $a=\frac{1}{2}$, $c=-1$, $q=g=1$, $T=2$, and the solution $P(\cdot)$ to the above Riccati equation is shown in Figure 1.
\begin{figure}[ht]
   	\centering
   	\includegraphics[width=0.7\linewidth]{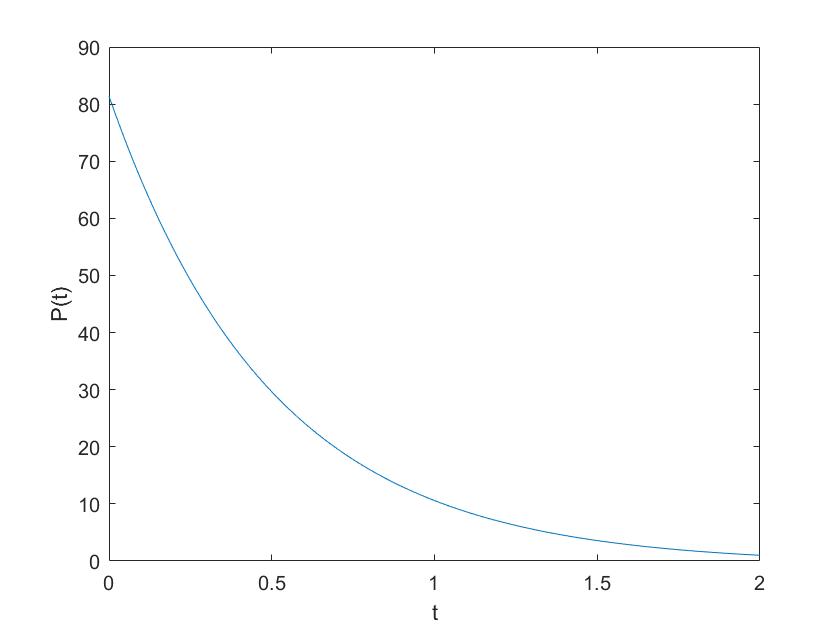}
   	\caption{The trajectory of $P(\cdot)$}
   	\label{fig:image}
\end{figure}
\end{Example}
			
\section{Conclusion}
			
In this paper, we have studied a class of zero-sum indefinite stochastic linear-quadratic Stackelberg differential games with two different drift uncertainties, which is more meaningful to describe the information/uncertainty asymmetry between the leader and the follower. By robust optimization method and decoupling technique, the robust Stackelberg equilibrium is obtained. The solvability of some matrix-valued Riccati equations are also discussed. Finally, the results are applied to the problem of market production and supply.

Other possible applications in practice are interesting research topic. Related problems with volatility uncertainties (\cite{Huang-Wang-Yong21}, \cite{Feng-Qiu-Wang24}) are also our future considerations.

\end{document}